\documentclass[letter]{amsart}
\usepackage{graphicx}

\usepackage{amsfonts}
\usepackage{amscd}
\usepackage{amssymb}
\usepackage{stmaryrd}
\usepackage{mathtools}
\usepackage{xypic}
\xyoption{all}
\usepackage{verbatim}
\usepackage{hyperref}
\usepackage{url}
\usepackage{tikz}
\usepackage{tikz-cd}
\usepackage{calligra}
\usetikzlibrary{shapes}
\usetikzlibrary{decorations.markings}
\usepackage{multirow}
\usepackage[savepos]{zref}
\newtheorem{thm}{Theorem}[section]
\newtheorem{corollary}[thm]{Corollary}
\newtheorem{lemma}[thm]{Lemma}
\newtheorem{proposition}[thm]{Proposition}

\newtheorem{example}[thm]{Example}
\newtheorem{conj}[thm]{Conjecture}

\theoremstyle{definition}
\newtheorem{definition}[thm]{Definition}
\theoremstyle{remark}
\newtheorem{remark}[thm]{Remark}
\numberwithin{thm}{section}
\numberwithin{equation}{section}
\usepackage[mathscr]{euscript}
\usepackage[T2A,T1]{fontenc}

  {\begin{description}%
    \setlength{\itemsep}{2.5pt}%
    \setlength{\parskip}{5pt}}%
  {\end{description}}

\DeclareMathOperator{\Supp}{Supp}

\DeclareMathOperator{\Star}{{St}}

\DeclareMathOperator{\fin}{fin}

\DeclareMathOperator{\ind}{ind}
\DeclareMathOperator{\cInd}{cInd}
\DeclareMathOperator{\Hom}{Hom}

\DeclareMathOperator{\Rep}{Rep}

\DeclareMathOperator{\Span}{Span}

\DeclareMathOperator{\Ker}{Ker}
\DeclareMathOperator{\Cok}{Cok}

\DeclareMathOperator{\Char}{char}

\DeclareMathOperator{\Id}{Id}
\DeclareMathOperator{\Res}{Res}

\newcommand{\From}{\colon}
\newcommand{\inar}{\ar@{^{(}->}}
\newcommand{\onar}{\ar@{->>}}

\renewcommand{\Im}{\mathrm{Im}}

\usepackage{accents}
\newlength{\dtildeheight}

\newcommand{\limdir}{\varinjlim}

\makeatletter
\newcommand{\raisemath}[1]{\mathpalette{\raisem@th{#1}}}
\newcommand{\raisem@th}[3]{\raisebox{#1}{$#2#3$}}
\makeatother

\newcommand{\Cat}[1]{ {\mathsf{#1}} }

\newcommand{\alg}[1]{\boldsymbol{\mathrm{#1}}}

\newcommand{\sheaf}[1]{{\mathscr{#1}}}

\newcommand{\NN}{\mathbb N}
\newcommand{\CC}{\mathbb C}

\newcommand{\ident}{\equiv}

\newcommand{\Into}{\hookrightarrow}
\newcommand{\Onto}{\twoheadrightarrow}
\newcommand{\To}{\rightarrow}

\newcommand{\isom}{\cong}

\DeclareMathOperator{\loc}{loc}

\makeatletter
\newcommand\@biprod[1]{%
  \vcenter{\hbox{\ooalign{$#1\prod$\cr$#1\coprod$\cr}}}}
\newcommand\biprod{\mathop{\mathpalette\@biprod\relax}\displaylimits}
\makeatother

\DeclareMathAlphabet{\mathcalligra}{T1}{calligra}{m}{n}

\makeatletter
\DeclareFontFamily{U}{MnSymbolA}{}
\DeclareFontShape{U}{MnSymbolA}{m}{n}{
    <-6>  MnSymbolA5
   <6-7>  MnSymbolA6
   <7-8>  MnSymbolA7
   <8-9>  MnSymbolA8
   <9-10> MnSymbolA9
  <10-12> MnSymbolA10
  <12->   MnSymbolA12}{}
\DeclareFontShape{U}{MnSymbolA}{b}{n}{
    <-6>  MnSymbolA-Bold5
   <6-7>  MnSymbolA-Bold6
   <7-8>  MnSymbolA-Bold7
   <8-9>  MnSymbolA-Bold8
   <9-10> MnSymbolA-Bold9
  <10-12> MnSymbolA-Bold10
  <12->   MnSymbolA-Bold12}{}
\DeclareSymbolFont{MnSyA}{U}{MnSymbolA}{m}{n}
\SetSymbolFont{MnSyA}{bold}{U}{MnSymbolA}{b}{n}

\DeclareRobustCommand{\overleftharpoon}{\mathpalette{\overarrow@\leftharpoonfill@}}
\DeclareRobustCommand{\overrightharpoon}{\mathpalette{\overarrow@\rightharpoonfill@}}
\def\leftharpoonfill@{\arrowfill@\leftharpoondown\mn@relbar\mn@relbar}
\def\rightharpoonfill@{\arrowfill@\mn@relbar\mn@relbar\rightharpoonup}

\DeclareMathSymbol{\leftharpoondown}{\mathrel}{MnSyA}{'112}
\DeclareMathSymbol{\rightharpoonup}{\mathrel}{MnSyA}{'100}
\DeclareMathSymbol{\mn@relbar}{\mathrel}{MnSyA}{'320}
\makeatother
\newcommand{\edge}[1]{\overrightharpoon{#1}}

\renewcommand{\NN}{{\mathcal N}}

\newcommand{\constsheaf}[1]{{\underline{#1}}}

\newcommand{\ltr}{\triangleleft}
\DeclareMathAlphabet{\mathcalligra}{T1}{calligra}{m}{n}

\DeclareMathOperator{\uni}{uni}
\DeclareMathOperator{\dist}{dist}

\DeclareMathOperator{\el}{ell}
\DeclareMathOperator{\pr}{pr}
\DeclareMathOperator{\minc}{mc}
\DeclareMathOperator{\cpt}{cpt}
\DeclareMathOperator{\res}{res}
\DeclareMathOperator{\GL}{GL}

\DeclareMathOperator{\coll}{coll}
\begin{document}

\title{Types and collapse for cuspidal representations of groups acting on trees}%
\author{Samuel Johnson and Martin H. Weissman}%
\date{\today}

\address{Dept. of Mathematics, University of California, Santa Cruz, CA 95064}
\email{weissman@ucsc.edu}%

% \thanks{}
% \subjclass{}%
% \keywords{}%
%\date{}%
%\dedicatory{}%
%\commby{}%
% ----------------------------------------------------------------
\begin{abstract}
In a previous paper, the second author proved that every supercuspidal representation of a rank-one $p$-adic group is induced from a compact-mod-center open subgroup.  The method was geometric, localizing representations to obtain equivariant sheaves on trees.  Here we provide two refinements.  The first is a geometric description of the inducing data, via a {\em geometrically minimal $K$-type}.  Second is a proof that the equivariant sheaves {\em collapse} onto injective sheaves.  The two notions of geometrically minimal $K$-types and collapsibility generalize to higher rank groups, suggesting a pair of conjectures.
\end{abstract}

\maketitle

\tableofcontents

\section{Introduction}

The {\em compact induction conjecture} is a folklore conjecture in the representation theory of $p$-adic groups.  A precise form is the following:  
\begin{conj}
\label{cind}
Let $F$ be a nonarchimedean local field, and let $\alg{G}$ be a connected reductive group defined over $F$.  Let $G = \alg{G}(F)$ be the resulting group of $F$-points, and let $Z$ be the center of $G$.  Let $(\pi, V)$ be an irreducible supercuspidal representation of $G$.  Then there exists a compact subgroup $K \subset G$, and a $K$-stable subspace $W \subset V$, such that
$$(\pi, V) = \cInd_{KZ}^G W \text{ as representations of } G.$$ 
(Note that $W$ is viewed as a representation of $KZ$ by having $Z$ act via the central character of $(\pi, V)$.)
\end{conj}

This form of the conjecture may differ from the traditional quick statement that $(\pi, V)$ is isomorphic to a represenation of the form $\cInd_{KZ}^G W$ (for some representation $W$ of $KZ$).  But it is equivalent:  if $(\pi, V)$ is isomorphic to $\cInd_{KZ}^G W$, then such an isomorphism identifies $W$ with a $K$-stable subspace of $V$.  And conversely, if $W$ is a $K$-stable subspace of $V$, then $W$ is a $KZ$-stable subspace of $V$, since $(\pi, V)$ has a central character.  Frobenius reciprocity gives a canonical $G$-intertwining map from $\cInd_{KZ}^G W$ to $(\pi, V)$, and our conjecture asks that this map is an isomorphism.  

We prefer the form of Conjecture \ref{cind} because it highlights the importance of finding a {\em $K$-stable subspace} within an irreducible supercuspidal representation of $G$.  A $K$-stable subspace $W \subset V$, whose compact induction $\cInd_{KZ}^G W$ coincides with $V$, is called a {\em minimal $K$-type} for the supercuspidal representation $(\pi, V)$.

Conjecture \ref{cind} has been proven in many settings, and we mention some highlights here.  When $\alg{G} = \alg{GL}_n$, the conjecture is proven in the monumental work of Bushnell-Kutzko \cite{BK}.  For classical groups (unitary, symplectic, or special orthogonal) in odd residual characteric, a proof is found is found in work of Stevens \cite{Stevens}.  These examples highlight one approach, which explicitly connects the arithmetic of lattices and orders to K-types.  

Another approach centers on the construction of supercuspidal representations by Yu \cite{JKY} (inspired by Adler \cite{Adler}) for a broad class of {\em tame} groups $\alg{G}$.  Breakthroughs of Kim \cite{JLK} and most recently Fintzen \cite{JF} now tell us that all supercuspidal irreducible representations arise from Yu's construction, under two natural conditions:  that the group $\alg{G}$ splits over a tamely-ramified extension of $F$, and that the residual characteristic $p$ does not divide the order of the (absolute) Weyl group of $\alg{G}$.  

Both of the above approaches rely of tameness in some sense.  The work of Stevens et al.~ on classical groups avoids residual characteristic two; the work of Fintzen et al.~ restricts the group, and the residual characteristic.  The wildest supercuspidal representations are at a far more exploratory stage.  Prototypes include various simple supercuspidal representations \cite{GR} and epipelagic representations \cite{RY}, but we seem far from a construction that would be conjecturally exhaustive.

A third approach is due to the second author in \cite{MWInd}, which proves the compact induction conjecture for groups of relative rank one.  This approach does not care about the residue characteristic (or characteristic) of the field $F$; it relies primarily on the geometry of the Bruhat-Tits building of $G$ and an analysis of equivariant sheaves introduced by Schneider and Stuhler \cite{SS}.  The proofs of \cite{MWInd} are limited to groups in which the building of $G$ is a tree -- a restrictive condition to be sure!  On the other hand, previous proofs of the compact induction conjecture required intricate analysis, even for $GL_2$ and $SL_2$.  

\subsection{Two definitions and two conjectures}

A (self-)critique of this geometric method of \cite{MWInd} is that the inducing data is unclear, and there is no indicated generalization to higher rank groups.  This article is meant to address these critiques by introducing very general notions of ``geometrically minimal $K$-type,'' and ``collapsible sheaves''.  We briefly introduce these notions here.

Let $(\pi, V)$ be an irreducible supercuspidal representation of $G$, and let $\sheaf{S}$ be the $G$-equivariant sheaf on the building $X$ associated to $(\pi, V)$ by Schneider and Stuhler in \cite{SS} (going sufficiently deep in their filtration).  Here, in this introduction and in \cite{SS}, we work with representations on complex vector spaces, and sheaves of complex vector spaces.  Schneider and Stuhler then identify $(\pi, V)$ with $H_c^0(X, \sheaf{S})$, as representations of $G$.  Let $B$ be a minimal ball in $X$ supporting a section of $\sheaf{S}$.  Let $\sigma$ be the center of $B$ (a facet of $X$).  Note that $H_B^0(X, \sheaf{S})$, the space of sections supported on $B$, is naturally a representation of $K = G_\sigma$.  We call this $H_B^0(X, \sheaf{S})$ a {\em geometrically minimal $K$-type} for $(\pi, V)$.
\begin{conj}[Geometrically minimal K-types]
\label{Ktype-conj}
The geometrically minimal $K$-type $H_B^0(X, \sheaf{S})$ is irreducible, and $(\pi, V) = \cInd_K^G( H_B^0(X, \sheaf{S}) )$.  The geometrically minimal $K$-type is also unique up to conjugation; if $B'$ is another minimal ball supporting a section of $\sheaf{S}$, then there exists $g \in G$ such that $g B = B'$.
\end{conj}

In a later section, we define a notion of (equivariant) collapse for sheaves on a tree; this generalizes naturally to the building as follows.  Given a pair of facets $\sigma, \tau$ such that $\sigma < \tau$ and $\dim(\tau) = \dim(\sigma) + 1$, a {\em basic elementary sheaf} on $X$ supported on $\{ \sigma, \tau \}$ is a sheaf with 1-dimensional stalks at $\sigma$ and $\tau$, whose restriction map $\Res_{\sigma, \tau}$ is an isomorphism.  An {\em elementary sheaf} on $X$ is a direct sum of basic elementary sheaves.  We say a sheaf $\sheaf{S}$ (equivariantly) {\em collapses} onto a sheaf $\sheaf{S}'$, if there is an (equivariant) elementary subsheaf $\sheaf{J} \subset \sheaf{S}$ such that $\sheaf{S}'$ is isomorphic to $\sheaf{S} / \sheaf{J}$.  
\begin{conj}[Collapse onto injectives]
\label{coll-conj}
The Schneider-Stuhler sheaf $\sheaf{S}$ can be equivariantly collapsed (through a finite sequence of equivariant collapses) onto an indecomposable equivariant {\em injective} sheaf.
\end{conj}
Under some finiteness conditions, e.g., via the admissibility of irreducible supercuspidals, injective sheaves are direct sums of constant sheaves supported on closures of facets.  Thus the conjecture above states that $\sheaf{S}$ can be collapsed onto an easily described sheaf -- one that arises from a constant sheaf on a closed facet by taking the direct sum of $G$-translates.  The two Conjectures \ref{Ktype-conj} and \ref{coll-conj} are closely linked; the latter says the Schneider-Stuhler sheaf collapses onto an ``induced'' sheaf; the conjecture about geometrically minimal $K$-types should follow by tracking sections through the collapse process.  

\medskip
In this article, we prove these conjectures for groups of relative rank one, i.e., groups where $X$ is a tree.  Along the way, we try to broaden the scope of the compact induction conjecture.  For this, we reprove key results of Schneider and Stuhler in the rank one case, where the geometry of trees makes everything easier.  Some may wish to study irreducible supercuspidal representations of $G$ with coefficient fields besides $\CC$.  Others may wish to study ``covering groups'' like metaplectic groups.  Others may wish to study disconnected groups.  We hope that our results accommodate such a breadth of groups, residue characteristics, characteristics, coefficient fields, central extensions, and disconnectedness... even as we stay within rank one.

\section{Conventions}

\subsection{Groups acting on trees}
\label{conventions}
Let $X$ be a tree, with vertex set $X^0$ and edge set $X^1$.  If $x \in X^0$ and $e \in X^1$, then we write $x < e$ to mean that $x$ is an endpoint of $e$.  If $x,y \in X^0$, then we write $[x,y]$ for the unique simple path joining $x$ and $y$.  The distance $\dist(x,y)$ is defined to be the number of edges in $[x,y]$.  

It will be convenient to choose an orientation on the tree $X$, by which we mean that each edge is given a direction.  We write $e = \edge{xy}$ to mean that $e$ is an edge directed from $x$ to $y$.  We represent this numerically with the {\em incidence numbers} $[e:x] = 1$ and $[e:y] = -1$.  

Let $G$ be a group acting on $X$.  This means that $G$ acts on the sets $X^0$ and $X^1$, and for all $g \in G$, $x \in X^0$, $e \in X^1$, 
$$x < e \text{ implies } g(x) < g(e).$$
It follows that the distance is $G$-invariant, i.e., $\dist(x,y) = \dist( g(x), g(y) )$ for all $x,y \in X^0$ and all $g \in G$.  For every vertex $x$ and edge $e$, we write $G_x$ and $G_e$ for their respective stabilizers.  

Throughout this article, we make the following \textbf{assumptions}.
\begin{description}
\item[Locally Finite]
$X$ is locally finite.  This means that for all $x \in X^0$, there are finitely many edges $e \in X^1$ for which $x < e$.
%\item[Pointwise Edge-stabilizing]
%If $x < e$, then $G_e \subset G_x$.  In other words, the stabilizer of an edge fixes the edge and its endpoints pointwise.  This is not a restrictive assumption, in practice, because we may always subdivide edges to make this assumption valid.  
\item[Edge transitive]
The set $X^1$ of edges is nonempty and $G$ acts transitively on $X^1$.  It follows that $G$ acts on $X^0$ with at most two orbits.  %The previous assumption then implies that $G$ acts on $X^0$ with precisely two orbits.
\end{description}

\subsection{Representations}

We fix a field $k$ in all that follows, and let $\Cat{Vec}$ be the category of $k$-vector spaces.  
\begin{definition}
A {\em representation} of $G$ is a pair $(\pi, V)$, where $V$ is a $k$-vector space, and $\pi \From G \To \GL(V)$ is a group homomorphism.
\end{definition}

Note that we do not assume that $V$ is finite-dimensional.  Let $\Cat{Rep}(G)$ be the resulting category of representations of $G$ and $G$-intertwining maps.  %Let $\Cat{For} \From \Cat{Rep}(G) \To \Cat{Vec}$ be the forgetful functor.

Let $H$ be a subgroup of $G$, and let $(\pi, V)$ be a representation of $G$.  We write $V^H$ for the $H$-invariant subspace of $V$.  Define
$$V[H] =   \Span_k \{ \pi(h) v - v : h \in H, v \in V \}.$$
Note that if $n \in N_G(H)$, then $\pi(n) V^H = V^H$ and $\pi(n) V[H] = V[H]$.  

We write $V_H$ for the $H$-coinvariant quotient:  $V_H = V / V[H]$.  The canonical maps $V^H \Into V \Onto V_H$ give a canonical linear map $c_H \From V^H \To V_H$.  
\begin{definition}
We say that $V$ is {\em $H$-characteristic} if the linear map $c_H$ is an isomorphism.  Equivalently, $V$ is {\em $H$-characteristic} if 
$$V = V^H \oplus V[H].$$
In this case, we write $\pr_H \From V \To V^H$ to be the composition of $V \Onto V_H$ with $c_H^{-1} \From V_H \To V^H$.  
\end{definition}
Note that $\pr_H^2 = \pr_H$; we say that $\pr_H$ is the projection onto $H$-fixed vectors.  When $V$ is $H$-characteristic, the decomposition $V = V^H \oplus V[H]$ is a decomposition of representations of $N_G(H)$, and $\pr_H$ is an $N_G(H)$-intertwining operator.
%\begin{lemma}
%\label{norm-proj}
%Let $(\pi, V)$ be a representation of $G$, and let $H$ be a subgroup of $G$.  Suppose that $V$ is $H$-characteristic, and $n \in N_G(H)$.  Then for all $v \in V$,
%$$\pr_H ( \pi(n) v ) = \pi(n) \pr_H(v).$$
%\end{lemma}
%\begin{proof}
%If $v \in V$, then let $w = v - \pr_H(v) \in V[H]$.  Thus $\pi(n) w \in V[H]$, and 
%$$\pi(n) v =  \pi(n) \pr_H(v) + \pi(n) w \in V^H + V[H].$$  
%Since $V = V^H \oplus V[H]$, we find that $\pi(n) \pr_H(v) = \pr_H( \pi(n) v)$.
%\end{proof}

When $H$ is a finite group, and $\Char(k) \nmid \# H$, all representations are $H$-characteristic and this projection map is well known: $\pr_H = (\# H)^{-1} \sum_{h \in H} \pi(h)$.  The following is a more general characterization.  It was provided by Claude Fable 5 (Anthropic's large language model in June 2026), in a form very close to what we write below.  The prompt is given in the appendix.
\begin{proposition}
\label{charchar}
Let $H$ be a finite group, and let $(\pi, V)$ be a representation of $H$.  Then $V$ is $H$-characteristic if and only if $V \isom T \oplus W$, where $T = T^H$ is a trivial representation of $H$ and $W$ is a representation with $W^H = W_H = 0$.
\end{proposition}
\begin{proof}
On one hand, if $V$ is $H$-characteristic, then $V = V^H \oplus V[H]$.  Note that this is a decomposition of representations of $H$, with $V^H$ a trivial representation of $H$.  As such, $V[H]^H \subset V^H \cap V[H] = 0$.  Moreover $V_H = (V^H)_H \oplus V[H]_H$; since $V$ is $H$-characteristic, we find $(V^H)_H = V_H$, and so $V[H]_H = 0$.  Hence we have the desired decomposition, with $T = V^H$ and $W = V[H]$.

Conversely, suppose that $V = T \oplus W$, where $T = T^H$ and $W^H = W_H = 0$.  Then we have $V^H = T^H \oplus W^H = T$.  Since $T[H] = 0$ and $W_H = 0$, we have $V[H] = T[H] \oplus W[H] = W$.  Since $V = T \oplus W$, we find $V = V^H \oplus V[H]$ as desired.
\end{proof}
Thus if we work in ``bad'' characteristic, e.g., $p \mid \#H$ and $\Char(k) = p$, then we must restrict to direct sums of trivial representations and representations that have no trivial representation in their socle or in their head.

\subsection{Equivariant sheaves on trees}

\begin{definition}
A sheaf on $X$ is the data of a $k$-vector space $\sheaf{S}_x$ for every vertex $x$, a $k$-vector space $\sheaf{S}_e$ for every edge $e$, and a $k$-linear map (``restriction'') $\Res_{x,e} \From \sheaf{S}_x \To \sheaf{S}_e$ for every pair $x < e$.  
\end{definition}

\begin{remark}
Such data forms a sheaf in the following sense.  The poset $(X, \leq)$ is endowed with the Alexandrov topology whose open sets are upward-closed.  A basis for the topology consists of singletons $\{ e \}$ for every edge, and stars $\Star(x) = \{ x \} \cup \{ e : e > x \}$ for every vertex $x$.  The reader may verify that a sheaf in the above definition coincides with a sheaf on $X$ in the Alexandrov topology.  Namely, $\sheaf{S}_x$ is a convenient shorthand for $\sheaf{S}(\Star(x))$.  
\end{remark}

Given such a sheaf $\sheaf{S}$, the compactly supported cochains are
$$C_c^0(X, \sheaf{S}) = \bigoplus_{x \in X^0} \sheaf{S}_x \text{ and } C_c^1(X, \sheaf{S}) = \bigoplus_{e \in X^1} \sheaf{S}_e.$$
Define the coboundary map $d \From C_c^0(X, \sheaf{S}) \To C_c^1(X, \sheaf{S})$ by the following:  if $x \in X^0$ and $f \in \sheaf{S}_x$,
$$(df)_e = [e:x] \cdot \Res_{x,e}(f) \text{ for all } e > x.$$
If $e \not > x$, then of course $(df)_e = 0$.  Note that our assumption of local finiteness is required here, for the definition of the coboundary map.

This defines the compactly-supported sheaf cohomology,
$$H_c^0(X, \sheaf{S}) = \Ker(d) \text{ and } H_c^1(X, \sheaf{S}) = \Cok(d).$$

A $G$-equivariant structure on a sheaf $\sheaf{S}$ is a family of isomorphisms $g^\ast \sheaf{S} \To \sheaf{S}$, for each $g \in G$, making an appropriate triangle commute for each composition in $G$.  A more concrete definition is below.
\begin{definition}
Let $\sheaf{S}$ be a sheaf on $X$.  A $G$-equivariant structure on $\sheaf{S}$ is a family of linear maps, $\rho(g,x) \From \sheaf{S}_x \To \sheaf{S}_{gx}$ and $\rho(g,e) \From \sheaf{S}_e \To \sheaf{S}_{ge}$, for all $g \in G$, $x \in X^0$, $e \in X^1$.  This family must be compatible with restriction and composition as described below.
\begin{itemize}
\item
$\Res_{gx, ge} \circ \rho_{g,x}  = \rho_{g,e} \circ \Res_{x,e}$ for all $x < e$ and all $g \in G$.
\item
$\rho_{g,hx} \circ \rho_{h,x} = \rho_{gh,x}$ and $\rho_{g,he} \circ \rho_{h,e} = \rho_{gh,e}$ for all $g,h \in G$, and all $x \in X^0$ and $e \in X^1$.
\end{itemize}
\end{definition}

Note that a $G$-equivariant structure on $\sheaf{S}$ yields an action of $G$ on $C_c^i(X, \sheaf{S})$ for $i \in \{ 0, 1 \}$.  This action is almost compatible with the coboundary operator $d$, in the following sense:  every element $g \in G$ is either orientation-preserving or orientation-reversing.  Thus if $f \in C_c^0(X, \sheaf{S})$, then $d(g f) = \pm g d(f)$.  Hence the kernel and cokernel of $d$ are also endowed with an action of $G$.  The cohomology $H_c^i(X, \sheaf{S})$ becomes a representation of $G$ for $i \in \{ 0, 1 \}$.  One may check that the cohomology (as a representation of $G$) does not depend on the choice of orientation.  

For example, observe that $H_c^0(X, \sheaf{S})$ can be identified with the following space of ``compactly supported global sections'':
$$\{ (f_x) \in \bigoplus_{x \in X^0} \sheaf{S}_x  \text{ such that } e = \edge{xy} \text{ implies } \Res_{x,e} f_x = \Res_{y,e} f_y \}.$$
This certainly does not reference the choice of orientation.  Moreover, if $f \in H_c^0(X, \sheaf{S})$, then we write $f_e$ unambiguously for the common value
$$f_e :=  \Res_{x,e} f_x = \Res_{y,e} f_y  \text{ at every edge } e = \edge{xy}.$$

Let $\Cat{Sh}(X)$ be the category whose objects are sheaves on $X$, and whose morphisms are systems of linear maps (for all $x$ and $e$) commuting with restrictions.  Let $\Cat{Sh}_G(X)$ be the category of such sheaves endowed with an equivariant structure.  Morphisms of equivariant sheaves are morphisms of sheaves which intertwine the equivariant structure.  Compactly-supported cohomology gives functors,
$$H_c^i(X, \bullet) \From \Cat{Sh}(X) \To \Cat{Vec} \text{ and } \Cat{Sh}_G(X) \To \Cat{Rep}(G).$$

\section{Localization}

When $G = \alg{G}(F)$ is a $p$-adic group (i.e., the $F$-points of a connected reductive group $\alg{G}$ over a local nonarchimedean field $F$), the landmark work of Schneider and Stuhler \cite[\S III.1]{SS} provides a functor from the category of smooth complex representations of $G$ to the category of $G$-equivariant sheaves on the building $X$ of $G$.  We call this functor {\em localization}; indeed, Schneider and Stuhler note that their ``constructions bear a certain resemblance to the Beilinson-Bernstein localization theory'' \cite[Intr., p.100]{SS}.  

The localization of Schneider and Stuhler begins with the construction of filtrations ``$U_F^{(e)}$'' of parahoric subgroups; around the same time, Moy and Prasad \cite{MP1} constructed their filtrations called ``${\mathcal P}_{x,r+}$''.  Moy and Prasad say that an irreducible representation of $G$ has {\em depth} $\leq r$ if it has vectors fixed by such a subgroup ${\mathcal P}_{x,r+}$.  Other filtrations may be more natural, and we would point towards the minimal congruent filtrations of Yu \cite{YuSmoothModels} as an example.  But over time, especially from the work of Meyer and Solleveld \cite{MS} and Savin and Bestvina \cite{BS},  it has become clear that the precise choice of filtration in Schneider-Stuhler is not as important as a few axiomatic properties.  

Here we introduce minimal axioms for filtrations, based largely on the axioms for consistent systems of idempotents in Meyer and Solleveld \cite[Definition 2.1]{MS}.  We call these systems of subgroups {\em plumbing systems} because they are used, in effect, to plumb the {\em depths} of representations of $G$.  But our setting is again abstract -- $G$ is a group acting on a tree $X$, following the conventions of Section \ref{conventions}.

\subsection{Plumbing systems}

\begin{definition}
Let $\NN = (N_x : x \in X^0)$ be a family of subgroups of $G$.  We say that $\NN$ is a plumbing system if the following axioms are satisfied.
\begin{description}
\item[Stabilizing]  If $x \in X^0$ then $N_x \subset G_x$.  If $e = \edge{xy}$, then $N_x \subset G_e$ and $N_y \subset G_e$.  
\item[Equivariance]  If $x \in X^0$ and $g \in G$, then $g N_x g^{-1} = N_{gx}$.  In particular, $N_x \ltr G_x$ for all vertices $x \in X^0$.
\item[Locality] If $e = \edge{xy}$ then $N_x N_y = N_y N_x$.  In this case, we define $N_e = N_x N_y$.  It follows that $N_e \ltr G_e$ for all edges $e \in X^1$.
\item[Convexity]  If $x,y,z \in X^0$ and $y \in [x,z]$ then $N_y \subset N_x N_z$.  
\end{description}
\end{definition}

There is a minor subtlety to consider, since the edge-stabilizer $G_e$ may differ from the {\em pointwise} stabilizer $G_e^\circ := G_x \cap G_y$.  But noting that $G_e^\circ = G_x \cap G_e = G_y \cap G_e$, we find that $N_x, N_y \subset G_e^\circ$ when $e = \vec{xy}$.  In particular, $N_x \subset G_y$ and $N_y \subset G_x$.  As stated above, $N_x$ and $N_y$ are normal subgroups of the full stabilizer $G_e$.

The following shows that convexity extends to edges.
\begin{lemma}
\label{convex-edges}
Let $\NN$ be a plumbing system.  If $x,z \in X^0$, and $e = \edge{pq}$ and $x, p, q, z$ lie along a simple path, then $N_e \subset N_x N_z$.  
\end{lemma}
\begin{proof}
If $n_e \in N_e$, then $n_e = n_p n_q$ for some $n_p \in N_p$ and $n_q \in N_q$.  
\begin{center}
\begin{tikzpicture}
\coordinate (X) at (0,0);
\coordinate (P) at (3,0);
\coordinate (Q) at (4,0);
\coordinate (Z) at (6,0);
\draw (0,0) -- (6,0);
\foreach \c/\l in {X/x,P/p,Q/q,Z/z}
{\filldraw[draw=black, fill=gray] (\c) circle (0.05) node[below] {$\l$};}
\draw (3.5,0) node[below] {$e$};
\end{tikzpicture}
\end{center}
Since $q \in [p,z]$, convexity implies $n_q = n_p' n_z$ for some $n_p' \in N_p$ and $n_z \in N_z$.  Thus we have $n_e = n_p n_p' n_z$.  Let $n_p'' = n_p n_p' \in N_p$.  Since $p \in [x,z]$, convexity now implies that $n_p'' = n_x n_z'$ for some $n_x \in N_x$ and $n_z' \in N_z$.  Thus we have 
$$n_e = n_p'' n_z = n_x n_z' n_z \in N_x N_z.$$
\end{proof}

The following three examples illustrate the scope of the article.

\begin{example}
\label{Ex1}
Let $G = \alg{G}(F)$ arise from a connected reductive group $\alg{G}$ of relative semisimple rank one over a nonarchimedean field $F$.  Examples of groups include $G = \alg{SL}_2(F)$ or $G = \alg{GL}_2(F)$ or $G = \alg{U}_3(F)$ for a nondegenerate 3-dimensional Hemitian space over a separable quadratic extension $E/F$.  Let $X$ be its (reduced) Bruhat-Tits building, which is a tree.  Let $r$ be a non-negative integer.  Then the system of Moy-Prasad subgroups $N_x = G_{x,r+}$ is a plumbing system, and $N_e = N_x N_y$ coincides with the Moy-Prasad subgroup $G_{e,r+}$.
\end{example}

\begin{example}
\label{Ex2}
With $G = \alg{G}(F)$ as in the previous example, let $\alg{\tilde G}$ be a central extension of $\alg{G}$ by $\alg{K}_2$, as studied by Brylinski and Deligne \cite{BD}.  Let $n$ be a positive integer, and let $\mu_n = \{ \zeta \in F : \zeta^n = 1 \}$.  Suppose $\# \mu_n = n$.  This gives a topological central extension,
$$1 \To \mu_n \To \tilde G \To G \To 1.$$
If $r$ is sufficiently large, then this extension $\tilde G$ splits over $G_{x,r+}$ for all $x$; moreover, there exists $s \geq r$ such that any two such splittings coincide when restricted to $G_{x,s+}$.  In this way, when $s$ is a sufficiently large integer, the system of Moy-Prasad subgroups $N_x = G_{x,s+}$ is a plumbing system for $\tilde G$ acting on $X$.  (Here the action of $\tilde G$ on $X$ factors through the quotient $G$.) 
\end{example}

\begin{example}
\label{Ex3}
With $\alg{G}$ as in the previous examples, suppose that $\alg{H}$ is a not-necessarily-connected reductive group over $F$ whose neutral component $\alg{H}^\circ$ equals $\alg{G}$.  Thus $H = \alg{H}(F)$ contains $G$ as a finite-index normal subgroup.  As shown in Bruhat-Tits, $H$ still acts on the building $X$.  When $r$ is a non-negative integer, we may consider the minimal congruent filtration $G_{x,r+}^{\minc}$ as defined by Yu in \cite{YuSmoothModels}.  As noted in \S9.4 of {\em loc.cit.}, this is defined via a congruence subgroup of a dilatation of the stabilizing group scheme of $x$ in $\alg{G}$.  Such constructions are respected by the action of the $H = \alg{H}(F)$, and so the minimal congruent filtration satisfies $H$-equivariance.  The root group decomposition of $G_{x,r+}^{\minc}$ gives locality and convexity.  Thus the minimal congruent filtration provides a plumbing system for the action of $H$ on $X$.
\end{example}

\subsection{Cogeneration and generation}

Let $(\pi, V)$ be a representation of $G$, and let $\NN$ be plumbing system.  Then we can look at the invariants and coinvariants of $V$ at each vertex $x \in X^0$.  There are canonical maps given by summation and projection,
$$\Sigma_\NN \From \bigoplus_{x \in X^0} V^{N_x} \To V \text{ and } P_\NN \From V \To \prod_{x \in X^0} V_{N_x}.$$
\begin{definition}
We say that $(\pi, V)$ is $\NN$-cogenerated if $P_\NN$ is injective, i.e., every nonzero vector $v \in V$ has nonzero image in some space of coinvariants $V_{N_x}$.  We say that $(\pi, V)$ is $\NN$-generated if $\Sigma_\NN$ is surjective, i.e., $V$ is spanned by its subspaces of invariants $V^{N_x}$.  
\end{definition}

For irreducible representations, generation and cogeneration is easily detected.
\begin{proposition}
Suppose that $(\pi, V)$ is an irreducible representation of $G$.  Then $(\pi, V)$ is $\NN$-cogenerated if and only if $V_{N_x} \neq 0$ for some $x \in X^0$.  Similarly, $(\pi, V)$ is $\NN$-generated if and only if $V^{N_x} \neq 0$ for some $x \in X^0$.
\end{proposition}
\begin{proof}
First, suppose that $(\pi, V)$ is $\NN$-generated.  Then $V$ is spanned by $\{ V^{N_x} : x \in X^0 \}$.  Hence $V^{N_x} \neq 0$ for some $x \in X^0$.  Conversely, suppose $V^{N_x} \neq 0$ for some $x \in X^0$.  Then irreducibility implies $V$ is spanned by $\{ \pi(g) v : v \in V^{N_x} \text{ and } g \in G \}$.  But when $v \in V^{N_x}$, we find $\pi(g) v \in V^{N_{gx}}$.  Hence $V = \sum_{g \in G} V^{N_{gx}}$, and so $V$ is $\NN$-generated.

Next, suppose that $(\pi, V)$ is $\NN$-cogenerated.  Then if $0 \neq v \in V$, its image is nonzero in some $V_{N_x}$.  Hence $V_{N_x} \neq 0$ for some $x \in X^0$.  Conversely, suppose that $V_{N_x}$ is nonzero for some $x \in X^0$.  Let $K = \Ker(P_\NN) \subset V$.  Then $K$ is a $G$-subrepresentation of $(\pi, V)$, so $K = 0$ or $K = V$.  Note that $K = V$ iff $V_{N_x} = 0$ for all $x \in X^0$.  Thus $K = 0$, so $P_\NN$ is injective, and $(\pi, V)$ is $\NN$-cogenerated.
\end{proof}

Define $\Cat{Rep}^\NN(G)$ be the category of $\NN$-generated representations, and $\Cat{Rep}_\NN(G)$ the category of $\NN$-cogenerated representations; both are additive $k$-linear full subcategories of $\Cat{Rep}(G)$.  The subcategory $\Cat{Rep}^\NN(G) \subset \Cat{Rep}(G)$ is stable under quotients, while $\Cat{Rep}_\NN$ is stable under sub-objects.  

\begin{remark}
In key examples, the category $\Cat{Rep}^\NN(G)$ is a Serre subcategory, i.e., closed under sub-objects and quotients.  This is crucial for some of the proofs in \cite{MS} and \cite{BS}.  Here we do not require such a strong hypothesis on $\NN$.
\end{remark}
 
%\begin{definition}
%We say that $(\pi, V)$ is $\NN$-characteristic if $V$ is $N_x$-characteristic for all $x \in X^0$.  In this case, $V$ is $N_e$-characteristic for all $e \in X^1$, and we define
%$$\pr_x = \pr_{N_x} \text{ and } \pr_e = \pr_{N_e},$$
%for the resulting idempotent projections on $V$ for all vertices and edges.
%\end{definition}

%\begin{proposition}
%If $(\pi, V)$ is $\NN$-characteristic, then $(\pi, V)$ is $\NN$-generated if and only if $(\pi, V)$ is $\NN$-cogenerated.
%\end{proposition}
%\begin{proof}
%IS THIS TRUE?  PROVE IT!  FALSE?  MODIFY THEN PROVE!
%\end{proof}

\subsection{Localization of representations}

Let $(\pi, V)$ be a representation of $G$, and let $\NN$ be a plumbing system.  We can use this to define a $G$-equivariant sheaf on $X$ as follows.  If $x \in X^0$, define $\sheaf{S}_x = V_{N_x}$.  Similarly, if $e = X^1$, define $\sheaf{S}_e = V_{N_e}$.  If $e = \edge{xy}$, recall that $N_e = N_x N_y = N_y N_x \ltr G_e$; thus $\sheaf{S}_e$ is naturally a quotient of both $\sheaf{S}_x$ and $\sheaf{S}_y$, providing surjective restriction maps $\sheaf{S}_x \Onto \sheaf{S}_e$ and $\sheaf{S}_y \Onto \sheaf{S}_e$.  This defines the sheaf $\sheaf{S}$.

The $G$-equivariant structure on $\sheaf{S}$ arises from the equivariance of $\NN$ and the representation $(\pi, V)$.  Namely, if $x \in X^0$ and $g \in G$, then $g N_x g^{-1} = N_{gx}$.  Hence $\pi(g)$ gives a well-defined map from $V / V[N_x]$ to $V / V[N_{gx}]$.  This is the desired equivariant structure $\rho(g,x) \From \sheaf{S}_x \To \sheaf{S}_{gx}$, and the same line of reasoning gives $\rho(g,e)$ for every edge $e$.  This completes the construction of the $G$-equivariant sheaf $\sheaf{S}$.
\begin{definition}
The $G$-equivariant sheaf $\sheaf{S}$ is called the {\em localization} of $(\pi, V)$, with respect to the plumbing system $\NN$.  This construction defines a right-exact $k$-linear additive functor,
$$\loc_{\NN} \From \Cat{Rep}(G) \To \Cat{Sh}_G(X).$$
(Right-exactness follows from the right-exactness of the coinvariant functor.)
\end{definition}

As the name suggests, vectors in the representation $(\pi, V)$ correspond to global sections of its localization $\sheaf{S} = \loc_{\NN}(\pi, V)$.  Indeed, if $f \in V$, we can consider its images $f_x \in \sheaf{S}_x = V_{N_x}$ and $f_e \in \sheaf{S}_e = V_{N_e}$ for every vertex $x$ and edge $e$.  By construction, it is clear that $\Res_{x,e} f_x = f_e$ whenever $x < e$.  In this way, $f \in H^0(X, \sheaf{S})$.  This provides a $G$-intertwining homomorphism,
$$\loc \From V \To H^0(X, \sheaf{S}).$$
Note that we use the same notation $\loc$ for the localization {\em functor} (from the category of representations to the category of sheaves), and for the localization {\em intertwining map} (from the particular representation $(\pi, V)$ to the global sections of the particular sheaf $\sheaf{S} = \loc_\NN(\pi, V)$).

The following just restates previous definitions.
\begin{proposition}
Let $\NN$ be a plumbing system, and let $\sheaf{S} = \loc_{\NN}(\pi, V)$ as before.  Then $(\pi, V)$ is $\NN$-cogenerated if and only if $\loc \From V \To H^0(X, \sheaf{S})$ is injective.
\end{proposition} 
\begin{proof}
The kernel of $\loc$ consists of those vectors $f \in V$ for which $f_x = 0$ for all $x \in X^0$.  But that is precisely the kernel of $P_\NN \From V \To \prod_x V_{N_x}$.  
\end{proof}

\subsection{Support}

Suppose that $(\pi, V)$ is a representation of $G$, and $\sheaf{S} = \loc_{\NN}(\pi, V)$.  When $f \in V$, we write $\Supp_\NN(f)$ to denote the support of the associated global section $\loc(f)$.  This is a closed subset of the tree $X$:
$$\Supp_\NN(f) = \{ x \in X^0 : f_x \neq 0 \} \cup \{ e \in X^1 : f_e \neq 0 \}.$$

Cogeneration can be interpreted using supports.
\begin{proposition}
The representation $(\pi, V)$ is $\NN$-cogenerated if and only if every nonzero vector $f \in V$ has nonempty support.
\end{proposition}

Supports cannot grow under intertwining operators.
\begin{lemma}
\label{cpt-image}
Suppose that $\phi \From (\pi, V) \To (\rho, W)$ is a $G$-intertwining operator.  Then for all $f \in V$, we have $\Supp_\NN(\phi(f)) \subset \Supp_\NN(f)$.
\end{lemma}
\begin{proof}
Let $\sheaf{T} = \loc_{\NN}(\rho, W)$ be the resulting equivariant sheaf, so that $\phi$ induces a morphism of equivariant sheaves,
$$\loc_\NN \phi \From \sheaf{S} \To \sheaf{T}.$$
If $f \in V$, then $\phi(f)_x \in \sheaf{T}_x$ coincides with $[\loc_\NN \phi]( f_x)$.  Hence $f_x = 0$ implies $\phi(f)_x = 0$.  The result follows.
\end{proof}

\begin{definition}
We say that $f \in V$ is an $\NN$-compact vector if $\Supp(f)$ is a compact subset of $X$.  Equivalently, since $\Supp_\NN(f)$ is closed, $f$ is $\NN$-compact if and only if $\Supp_\NN(f)$ is finite.
\end{definition}

\begin{proposition}
The $\NN$-compact vectors form a $G$-subrepresentation of $(\pi, V)$.  Intertwining operators send compact vectors to compact vectors.
\end{proposition}
\begin{proof}
This follows from Lemma \ref{cpt-image}, and basic facts about support.  For example, $\Supp_\NN( \pi(g) f) = g \cdot \Supp_\NN(f)$, and $\Supp_\NN(f_1 + f_2) \subset \Supp_\NN(f_1) \cup \Supp_\NN(f_2)$.
\end{proof}

\begin{definition}
We say that $(\pi, V)$ is $\NN$-compact if for all $v \in V$, $\Supp_\NN(v)$ is compact.  Let $\Cat{Rep}_\NN^{\cpt}(G)$ be the full subcategory of $\Rep(G)$ consisting of $\NN$-compact $\NN$-cogenerated representations.
\end{definition}

The importance of this definition is the following observation.
\begin{proposition}
\label{inj-loc}
Suppose that $(\pi, V)$ is an $\NN$-compact and $\NN$-cogenerated representation of $G$.  Then localization gives an injective homomorphism,
$$\loc \From V \Into H_c^0(X, \sheaf{S}),$$
\end{proposition}

\subsection{Localization via invariants}

While localization makes sense as a functor from $\Cat{Rep}(G)$ to $\Cat{Sh}_G(X)$ in the broad generality of the previous sections, it does not have particularly nice properties until we make a crucial assumption about characteristic.
\begin{definition}
We say that $(\pi, V)$ is $\NN$-characteristic if $V$ is $N_x$-characteristic for all $x \in X^0$.  In this case, we define $\pr_x = \pr_{N_x}$ for the resulting idempotent projections on $V$ for all vertices $x \in X^0$.
\end{definition}
Recall that if $e = \edge{xy}$, then $N_e = N_x N_y = N_y N_x \ltr G_e$.  The following seems a bit subtle, perhaps because we grew up studying invariants rather than coinvariants.
\begin{proposition}
If $(\pi, V)$ is $\NN$-characteristic, then $V$ is $N_e$-characteristic for all $e \in X^1$.  Writing $\pr_e$ for $\pr_{N_e}$, we have
$$\pr_e = \pr_x \pr_y = \pr_y \pr_x.$$
\end{proposition}
\begin{proof}
Since $N_e = N_x N_y = N_y N_x$, it is a quick exercise to show that
\begin{align*}
V^{N_e} &= V^{N_x} \cap V^{N_y}; \\
V[N_e] &= V[N_x] + V[N_y].
\end{align*}

To prove that $(\pi, V)$ is $N_e$-characteristic, we prove that $V = V^{N_e} \oplus V[N_e]$.  Begin with $v \in V$.  Since $V$ is $N_x$-characteristic, there exist unique $v_x \in V^{N_x}$ and $s_x \in V[N_x]$ such that 
$$v = v_x + s_x.$$
This $v_x$ equals $\pr_x(v)$.  Since $V$ is $N_y$-characteristic, there exist unique $v_{xy} \in V^{N_y}$ and $s_{xy} \in V[N_y]$ such that
$$v_x = v_{xy} + s_{xy}.$$

Since $N_x \subset G_x \cap G_y$, we find that $N_x$ normalizes $N_y$.  Hence if $n \in N_x$, then $\pi(n) v_{xy} \in V^{N_y}$ and $\pi(n) s_{xy} \in V[N_y]$.  Since $\pi(n) v_x = v_x$, we have
$$v_x = \pi(n) v_x = \pi(n) v_{xy} + \pi(n) s_{xy}$$
and the uniqueness implies that $\pi(n) v_{xy} = v_{xy}$.  Therefore $ v_{xy} = \pr_y(\pr_x(v)) \in V^{N_x} \cap V^{N_y}$.  Putting this together, we find that 
$$v = v_{xy} + s_{xy} + s_x \in V^{N_e} + V[N_y] + V[N_x] = V^{N_e} + V[N_e].$$
This proves that 
$$V = V^{N_e} + V[N_e].$$

It remains to prove that $V^{N_e} \cap V[N_e] = 0$.  Let us suppose $v \in V^{N_e} \cap V[N_e] = V^{N_e} \cap (V[N_x] + V[N_y])$.  Then we have $v = w_x + w_y$ for some $w_x \in V[N_x]$ and $w_y \in V[N_y]$.  Since $\pr_y(w_y) = 0$, we have
$$v = \pr_x \pr_y(v) = \pr_x \pr_y(w_x).$$
Since $w_x \in V[N_x]$, and every element $n \in N_x$ normalizes $N_y$, we find that $\pr_y(w_x) \in V[N_x]$.  Thus $v = \pr_x(\pr_y(w_x)) = 0$.  
\end{proof}

Thus when $(\pi, V)$ is $\NN$-characteristic, there are canonical isomorphisms,
$$\sheaf{S}_x = V_{N_x} \ident V^{N_x} \text{ and } \sheaf{S}_e = V_{N_e} \ident V^{N_e}.$$
Moreover, these canonical isomorphisms intertwine the restriction maps with the projection maps as below.
\begin{lemma}
If $x < e$, then the following diagram commutes.
$$
\begin{tikzcd}
V^{N_x} \arrow[r, "\pr_e"] \arrow[d, "c_x"] & V^{N_e} \arrow[d, "c_e"] \\
V_{N_x} \arrow[r, "\Res_{x,e}"]  & V_{N_e}
\end{tikzcd}
$$
\end{lemma}
\begin{proof}
Suppose that $v \in V^{N_x}$.  Then the canonical map $c_x$ sends $v$ to its coset $v + V[N_x]$ in $V_x = V / V[N_x]$; the restriction map $\Res_{x,e}$ sends this to its coset modulo $V[N_e]$.  Thus 
$$\Res_{x,e} c_x(v) = v + V[N_e].$$
On the other hand, this coset is represented precisely by $\pr_e(v)$, showing that the diagram commutes.
\end{proof}

In this way, we find
\begin{proposition}
When $(\pi, V)$ is $\NN$-characteristic, the sheaf $\sheaf{S}$ is canonically isomorphic to the sheaf whose stalks are the invariants $V^{N_x}$ and $V^{N_e}$, and whose restriction maps are the projections $\pr_e \From V^{N_x} \To V^{N_e}$.
\end{proposition} 

\subsection{Localization for cuspidal representations}

Suppose that $G$ is a second-countable, locally compact, totally disconnected group -- an $\ell$-group in the terminology of Bernstein \cite{Bernstein}.  This covers the key applications we have in mind, to $p$-adic groups including covering groups and disconnected groups as in Examples \ref{Ex1}, \ref{Ex2}, \ref{Ex3}.  
We assume that as $G$ acts on the tree $X$, the center $Z = Z(G)$ acts trivially on $X$.  We also assume that a plumbing system $\NN$ is chosen so that all subgroups $N_x$ are open.

Recall that a representation $(\pi, V)$ of $G$ is called smooth if every vector $v \in V$ has open stabilizer in $G$.  For convenience, in this section, we fix a smooth character $\chi \From Z \To k^\times$, and work with smooth representations having central character $\chi$.  In other words, we assume that for all $z \in Z$ and all $v \in V$,
$$\pi(z) v = \chi(z) \cdot v.$$

There are various notions of ``cuspidal'' in the literature, which are often equivalent.  One relates to the vanishing of Jacquet modules, and thus relies on some reductive group structure.  The other relates to matrix coefficients, as defined below.
\begin{definition}
Let $(\pi, V)$ be a smooth representation of $G$ on a $k$-vector space.  Let $\tilde V$ be the contragredient representation, i.e., the space of smooth vectors in the linear dual $\Hom_k(V, k)$.  If $v \in V$ and $\lambda \in \tilde V$, we define the {\em matrix coefficient} $m_{\lambda, v} \From G \To k$ by
$$m_{\lambda, v}(g) = \lambda \left( \pi(g)^{-1} v \right).$$
\end{definition}

Now we define supercuspidal representations via their matrix coefficients.
\begin{definition}
We say that a smooth representation $(\pi, V)$ is {\em supercuspidal} if for all $v \in V$, $\lambda \in \tilde V$, the function $m_{\lambda, v}$ is compactly supported, modulo $Z$.
\end{definition}

Suppose that $(\pi, V)$ is a smooth $\NN$-characteristic representation of $G$.  If $f \in V$, and $x \in X^0$, we write $f_x = \pr_x(f) \in \sheaf{S}_x = V^{N_x}$.  More generally, we have
$$\pi(g)^{-1} f_{gx} = \pi(g)^{-1} \pr_{gx} f = \pr_x \pi(g)^{-1} f \in V^{N_x}.$$
The following key lemma, its proof and its consequences, are slight adaptations of the work of Bernstein in \cite[Theorem 6]{Bernstein}.

\begin{lemma}
\label{findim-lemma}
Let $(\pi, V)$ be a smooth supercuspidal representation of $G$ with central character $\chi$.  Assume that $(\pi, V)$ is $\NN$-characteristic.  Then for all $f \in V$ and all $x \in X^0$, we have
$$\dim \left( \Span_k \{ \pr_x \pi(g)^{-1} f : g \in G \} \right) < \infty.$$
\end{lemma}
\begin{proof}
Fixing $f$ and $x$ as above, define $\delta(g) = \pr_x \pi(g)^{-1} f \in V^{N_x}$ for all $g \in G$.  Now define
$$\Delta = \{ \delta(g) : g \in G \} \subset V^{N_x}.$$
The goal is to prove that $\dim \Span_k(\Delta) < \infty$.  

Let $F \subset \Delta$ be a maximal linearly independent subset.  There exists a corresponding subset $S \subset G$ for which the map $g \mapsto \delta(g)$ is a bijection from $S$ to $F$.  Thus we have
$$F = \{ \delta(g) : g \in S \} \subset \Delta \subset V^{N_x}.$$
By linear independence, there exists a linear functional on $V^{N_x}$ which takes the value $1$ on every element of $F$.  Pulling this back via $\pr_x$ yields a linear functional $\lambda \From V \To k$ satisfying
$$\lambda = \lambda \circ \pr_x \text{ and } \lambda(\delta(g)) = 1 \text{ for all } g \in S.$$

Since $\lambda = \lambda \circ \pr_x$, we have $\lambda \in \tilde V$.  Thus we can consider the matrix coefficient,
\begin{align*}
m_{\lambda, f}(g) &= \lambda( \pi(g)^{-1} f) \\
&= \lambda( \pr_x \pi(g)^{-1} f) = \lambda(\delta(g)).
\end{align*}
Since $(\pi, V)$ is supercuspidal, the matrix coefficient $m_{\lambda, f}$ is compactly supported modulo $Z$.  Hence there exists a compact subset $K \subset G$ such that $S \subset K Z$.  

Note that $\delta(gn) = \delta(g)$ for all $n \in N_x$, and $\delta(g z) = \chi(z)^{-1} \cdot \delta(g)$ for all $z \in Z$.  Hence the linear independence of $F$ implies that the elements of $S$ lie in distinct $N_x Z$ cosets.  Since $N_x$ is open, and $K$ is compact, there are only finitely many $N_x Z$ cosets in $K N_x Z$.  Hence $S$ is finite.  Hence $F$ is finite, and hence $\Span_k(\Delta) = \Span_k(F)$ is finite-dimensional.  
\end{proof}
  
\begin{thm}
\label{cusp-cpt}
Let $(\pi, V)$ be a smooth supercuspidal representation of $G$ with central character $\chi$.  Assume that $(\pi, V)$ is $\NN$-characteristic.  Then $(\pi, V)$ is $\NN$-compact.
\end{thm}
\begin{proof}
Suppose that $f \in V$, and identify $f$ with its image in $H_c^0(X, \sheaf{S})$.  Thus $f_x = \pr_x(f)$ for all $x \in X^0$.  Recall that $G$ acts transitively on the edges of $X$, and let $e = \vec {xy}$ be an edge.  Thus every vertex is in the $G$-orbit of $x$ or in the orbit of $y$.  Define subspaces of $\sheaf{S}_x = V^{N_x}$ and $\sheaf{S}_y = V^{N_y}$,
$$W_x = \Span_k \{ \pr_x \pi(g)^{-1} f : g \in G \} \text{ and } W_y = \Span_k \{ \pr_y \pi(g)^{-1} f : g \in G \}.$$
These are finite-dimensional, by Lemma \ref{findim-lemma}.  

Let $\Lambda_x$ and $\Lambda_y$ be bases of $\Hom_k(W_x, k)$ and $\Hom_k(W_y, k)$, respectively.  If $\lambda \in \Lambda_x$, extend $\lambda$ to $V^{N_x}$ arbitrarily, and then extend by zero to $V = V^{N_x} \oplus V[N_x]$.  Similarly, extend all functionals $\lambda \in \Lambda_y$ to $V$.  

For all $\lambda \in \Lambda_x$, we have $m_{\lambda, f}(g) = \lambda( \pr_x \pi(g)^{-1} f) = \lambda(\pi(g)^{-1} f_{gx})$.  Hence  for all $g \in G$, we have
$$f_{gx} = 0 \text{ iff } m_{\lambda, f}(g) = 0 \text{ for all } \lambda \in \Lambda_x.$$
In this way we find that
$$f_{gx} \neq 0 \text{ implies } g \in \Supp(m_{\lambda, f}) \text{ for some } \lambda \in \Lambda_x.$$
And similarly,
$$f_{gy} \neq 0 \text{ implies } g \in \Supp(m_{\lambda, f}) \text{ for some } \lambda \in \Lambda_y.$$
As $\Lambda_x$ and $\Lambda_y$ are finite, we find that
$$\bigcup_{\lambda \in \Lambda_x} \Supp(m_{\lambda, f}) \text{ and } \bigcup_{\lambda \in \Lambda_y} \Supp(m_{\lambda, f}) \text{ are compact modulo } Z.$$

Thus there exists a compact subset $K \subset G$ such that $f_p \neq 0$ implies $p \in K Z \cdot x$ or $p \in K Z \cdot y$.  Since the stabilizers of $x$ and $y$ are open subgroups of $G$ containing $Z$, it follows that $f_p \neq 0$ for finitely many points in the orbit of $x$ and finitely many points in the orbit of $y$.  Hence $\Supp_\NN(f)$ is compact.
\end{proof}

The other striking consequence of Lemma \ref{findim-lemma} is the following central result.
\begin{thm}
Let $(\pi, V)$ be a smooth irreducible supercuspidal representation of $G$ with central character $\chi$.  Assume that $(\pi, V)$ is $\NN$-characteristic.  Then $V^{N_x}$ is finite-dimensional for every $x$, i.e., the sheaf $\sheaf{S}$ has finite-dimensional stalks.
\end{thm}
\begin{proof}
Let $f$ be any nonzero vector in $V$.  By irreducibility, we have $V = \Span_k \{ \pi(g)^{-1} f : g \in G \}$.  Therefore, we find that
$$\Span_k \{ \pr_x \pi(g)^{-1} f : g \in G \} = \pr_x(V) = V^{N_x}.$$
Lemma \ref{findim-lemma} completes the proof.
\end{proof}

\subsection{The localization theorem}

\begin{lemma}
\label{xez-lem}
Suppose that $(\pi, V)$ is $\NN$-characteristic.  Let $x$ and $z$ be distinct vertices in $X$, and let $e = \edge{xy}$ be the first edge in the unique simple path from $x$ to $z$.  Then
$$\pr_z \pr_x = \pr_z \pr_y \pr_x = \pr_z \pr_e.$$
\end{lemma}
\begin{proof}
Since $\pr_y \pr_x = \pr_e$, it suffices to begin with $v \in V^{N_x}$ and prove that $\pr_z(v) = \pr_z \pr_y(v)$.  Begin by noting that there exists a unique $s_y \in V[N_y]$ for which
\begin{equation}
\label{v-pry}
v = \pr_y(v) + s_y.
\end{equation}
Since $V \in V^{N_x}$, and $N_x$ normalizes $N_y$, we find that $\pr_y(v) \in V^{N_x}$ too.  Hence $s_y \in V^{N_x}$.  Since $s_y \in V[N_y]$ and $N_y \subset N_z N_x$, we find that $s_y \in V[N_z] + V[N_x]$.  But then we find $\pr_x(s_y) = s_y$, and so $s_y \in V[N_z]$.  

Taking $\pr_z$ of both sides of Equation \ref{v-pry}, we now have
$$\pr_z(v) = \pr_z \pr_y(v) + \pr_z(s_y) = \pr_z \pr_y(v).$$
That completes the proof.
\end{proof}

The proof of the following Theorem \ref{localization} is close in spirit to Meyer-Solleveld, who use a similar alternating sum for a support projection in \cite[Theorem 2.12]{MS}.  But we are primarily interested in the {\em cohomological} resolution of $\NN$-compact (e.g., supercuspidal) representations, while Meyer and Solleveld \cite{MS}, and Schneider and Stuhler \cite{SS}, and Vigneras \cite{Vigneras} treat the homological resolution as central.  Those authors prove a result akin to Theorem \ref{localization} by an intricate argument which begins with a homological resolution via cosheaves, and extends to the compactified building, and applies a duality theorem.  Our proof is direct, and cohomological throughout.
\begin{thm}
\label{localization}
Let $\NN$ be a localizing family of subgroups of $G$.  Let $(\pi, V)$ be an $\NN$-compact, $\NN$-cogenerated, and $\NN$-characteristic representation of $G$.  Then localization gives an isomorphism of representations of $G$,
$$\loc \From (\pi, V) \xrightarrow{\sim} H_c^0(X, \sheaf{S}).$$
\end{thm}
\begin{proof}
Since $(\pi, V)$ is $\NN$-cogenerated and $\NN$-compact, Proposition \ref{inj-loc} shows that localization is an {\em injective} $G$-intertwining map
$$\loc \From V \Into H_c^0(X, \sheaf{S}).$$

For surjectivity, begin with a compactly supported global section $f \in H_c^0(X, \sheaf{S})$.  Since $(\pi, V)$ is $\NN$-characteristic, we identify coinvariants with invariants, so 
$$f_x \in \sheaf{S}_x \ident V^{N_x} \text{ and } f_e \in \sheaf{S}_e \ident V^{N_e},$$ 
for all vertices and edges.

Define a vector $v \in V$ by the formula
$$v = \sum_{x \in X^0} f_x - \sum_{e \in X^1} f_e.$$
The compact support of $f$ guarantees the finiteness of the sums above.

We claim that $\loc(v) = f$.  Indeed, at any vertex $z$, we have
\begin{equation}
\label{L_Cancel}
\loc(v)_z = \sum_{x \in X^0}  \pr_z f_x - \sum_{e \in X^1} \pr_z f_e.
\end{equation}
If $x \neq z$, then let $e$ be the first edge in the unique path from $x$ to $z$.  (Here we use the fact that $X$ is a tree!)  Then by Lemma \ref{xez-lem}, we have
\begin{equation}
\label{proj_xz}
\pr_z f_x = \pr_z \pr_x f_x = \pr_z \pr_e f_x = \pr_z f_e.
\end{equation}

Hence each term $\pr_z f_x$ in \eqref{L_Cancel} cancels a corresponding term $\pr_z f_e$, when we pair off each vertex $x$ with its adjacent edge $e$ pointing towards $z$.  The only unpaired term is $z$ itself.  Thus we find
$$\loc(v)_z = \pr_z f_z = f_z \text{ for all } z \in X^0.$$
Hence $\loc(v) = f$, completing the proof that $\loc$ is surjective.   
\end{proof}

\section{Categories of sheaves}

Let $X$ be a locally finite tree as before.  Let $\Cat{Sh}^{\fin}(X)$ be the full subcategory of $\Cat{Sh}(X)$ consisting of sheaves whose stalks are all finite-dimensional.  Similarly, write $\Cat{Sh}_G^{\fin}(X)$ for the category of $G$-equivariant sheaves whose stalks are finite-dimensional.  Our goal in this section is to understand basic properties of these abelian categories and their injective objects.

\subsection{Decomposition}

We can view both categories $\Cat{Sh}^{\fin}(X)$ and $\Cat{Sh}_G^{\fin}(X)$ as categories of presheaves (i.e., functor categories) in the following way:  Define $\Cat{Q}$ to be the small category with object set $X^0 \sqcup X^1$, and morphisms $\lambda_{x,e} \From x \To e$ for every instance that $x < e$ (and identity morphisms of course).  In other words, $\Cat{Q}$ is the category associated to the poset $(X, \leq)$.  Then a sheaf on $X$ is precisely the same as a functor from $\Cat{Q}$ to $\Cat{Vec}$.  The category $\Cat{Sh}(X)$ is isomorphic to the category $[\Cat{Q},\Cat{Vec}]$ of functors from $\Cat{Q}$ to the category of vector spaces.  This is also known as the category of {\em presheaves (of vector spaces) on $\Cat{Q}$}.  Similarly, $\Cat{Sh}^{\fin}(X)$ is isomorphic to $[\Cat{Q}, \Cat{Vec}^{\fin}]$.     

In the equivariant setting, we can build a category which incorporates the poset structure of $X$ and the action of $G$.  Let $\Cat{Q}_G$ be the small category with the same object set $X^0 \sqcup X^1$ as before.  The morphisms in $\Cat{Q}_G$ belong to three families.  
\begin{itemize}
\item
As before, we include morphisms $\lambda_{x,e} \From x \To e$ whenever $x < e$. 
\item
We capture the group action by including morphisms $\alpha_{g,x} \From x \To gx$ and $\alpha_{g,e} \From e \To ge$ whenever $x \in X^0$, $e \in X^1$, and $g \in G$.  When $g = \Id$, we identify $\alpha_{\Id, x} = \Id_x$ and $\alpha_{\Id, e}$ with $\Id_e$. 
\item
We include morphisms that mix poset structure and group action:  these are morphisms $\beta_{g,x,e} \From x \To ge$, whenever $x < e$ and $g \in G$.  When $g = \Id$, we identify $\beta_{\Id, x, e}$ with $\lambda_{x,e}$.    
\end{itemize}
Compositions are almost self-evident.  They are determined by the following rules.
\begin{itemize}
\item
If $g,h \in G$, and $x$ is a vertex, then $\alpha_{h, gx} \circ \alpha_{g,x} = \alpha_{hg,x}$.  Similarly, if $e$ is an edge, then $\alpha_{h, ge} \circ \alpha_{g,e} = \alpha_{hg, e}$.  
\item
If $g \in G$, and $x < e$, then we identify
$$\beta_{g,x,e} = \lambda_{gx, ge} \circ \alpha_{g,x} = \alpha_{g,e} \circ \lambda_{x,e}.$$
\end{itemize}

This category is cooked up so that the following holds.
\begin{proposition}
The category $\Cat{Sh}_G(X)$ is isomorphic to the category $[\Cat{Q}_G, \Cat{Vec}]$ of functors from $\Cat{Q}_G$ to $\Cat{Vec}$, i.e., the category of presheaves on $\Cat{Q}_G$.  Similarly, $\Cat{Sh}_G^{\fin}(X)$ is isomorphic to $[\Cat{Q}_G, \Cat{Vec}^{\fin}]$. 
\end{proposition}
 
 Framing our (equivariant) sheaves as presheaves in this way, we can apply a vast general literature.  For example, in \cite[Theorem 1.1]{BCB}, Botnan and Crawley-Boevey prove the following Atiyah-Krull-Schmidt type result.
\begin{thm}
Let $\Cat{M}$ be the category of functors from a small category to the category $\Cat{Vec}^{\fin}$.  Then every object of $\Cat{M}$ decomposes as a direct sum of indecomposable objects with local endomorphism ring.
\end{thm}

It follows that (equivariant) sheaves on $X$ admit such a decomposition.
\begin{corollary}
Every sheaf in $\Cat{Sh}^{\fin}(X)$ decomposes into a direct sum of indecomposable sheaves.  Moreover, every equivariant sheaf in $\Cat{Sh}_G^{\fin}(X)$ decomposes into a direct sum of indecomposable $G$-equivariant sheaves.
\end{corollary}

\subsection{Elementary injectives and adjoint functors}

Since $\Cat{Sh}(X)$ can be viewed as the category of presheaves $[\Cat{Q}, \Cat{Vec}]$ -- there is a general machinery which provides injective objects and resolutions.  We refer to the Stacks Project \cite[\href{https://stacks.math.columbia.edu/tag/01DJ}{Tag 01DJ}]{stacks-project}, and describe the main ideas here.  

Let $\Cat{Q}_\circ$ be the ``discrete'' category with object set $X^0 \sqcup X^1$ as before, but with only identity morphisms.  A functor in $[\Cat{Q}_\circ, \Cat{Vec}]$ consists of just the data of a vector space for every vertex and edge of the tree $X$.  We call such data a {\em discrete sheaf}, and let $\Cat{Sh}_\circ(X) = [\Cat{Q}_\circ, \Cat{Vec}]$ be the category of such discrete sheaves (and linear maps among them).  Every object of $\Cat{Sh}_\circ(X)$ is injective and projective.   

The inclusion of categories $\Cat{Q}_\circ \Into \Cat{Q}$ yields a restriction functor $\iota^\circ \From \Cat{Sh}(X) \To \Cat{Sh}_\circ(X)$.  Explicitly, this is the functor which takes a sheaf on $X$ and forgets all the $\Res_{x,e}$ maps.  This functor has a {\em right} adjoint, as described at \cite[\href{https://stacks.math.columbia.edu/tag/00XF}{Tag 00XF}]{stacks-project}, which we call $\iota_\circ$,
$$\iota_\circ \From \Cat{Sh}_\circ(X) \To \Cat{Sh}(X).$$
Exactness of $\iota^\circ$ implies that $\iota_\circ$ transforms injective objects to injective objects \cite[\href{https://stacks.math.columbia.edu/tag/015Z}{Tag 015Z}]{stacks-project}.  This is quite helpful, as it allows us to produce injective resolutions of sheaves from injective resolutions of discrete sheaves (a category where all objects are injective).
\begin{proposition}
Let $\sheaf{S}$ be a sheaf on $X$.  Then $\sheaf{S} \hookrightarrow \iota_\circ \iota^\circ \sheaf{S}$ is an embedding of $\sheaf{S}$ into an injective sheaf.\end{proposition}

Thankfully, the functor $\iota_\circ$ can be described explicitly, via \cite[\href{https://stacks.math.columbia.edu/tag/00XF}{Tag 00XF}]{stacks-project}.  Since every discrete sheaf is a direct sum of discrete sheaves supported at a single vertex or edge, it suffices to describe $\iota_\circ$ for such discrete sheaves with isolated supports.
\begin{proposition}
Let $E$ be a $k$-vector space.  Let $x$ be a vertex and let $\constsheaf{E}_x^\circ$ be the discrete sheaf supported at $x$ with stalk $E$.  Then $\iota_\circ \constsheaf{E}_x^\circ = \constsheaf{E}_x$, the skyscraper sheaf supported at $x$ with stalk $E$.  

Similarly, let $e = \edge{xy}$ be an edge, and $\constsheaf{E}_e^\circ$ the discrete sheaf supported at $e$ with stalk $E$.  Then $\iota_\circ \constsheaf{E}_e^\circ = \constsheaf{E}_{\bar e}$.  This is the constant sheaf supported on the closure $\bar e = \{ e, x, y \}$ of the edge.
\end{proposition}
\begin{proof}
This follows directly from the general result of \cite[\href{https://stacks.math.columbia.edu/tag/01DJ}{Tag 01DJ}]{stacks-project}.
\end{proof}

\begin{corollary}
\label{indec-inj}
The sheaves $\constsheaf{k}_x$ and $\constsheaf{k}_{\bar e}$ are indecomposable injective objects in the category $\Cat{Sh}(X)$.  Every indecomposable injective sheaf on $X$ is isomorphic to such a sheaf.
\end{corollary}
\begin{proof}
These sheaves are injective, as they arise from $\iota_\circ \constsheaf{k}_x^\circ$ and $\iota_\circ \constsheaf{k}_e^\circ$, respectively.  It is clear that they are indecomposable, if one spends a moment trying to decompose them.  We prove that all indecomposable injective sheaves on $X$ arise in this way.  

So let $\sheaf{I}$ be an indecomposable injective sheaf.  Then there is an inclusion of injective sheaves,
$$0 \To \sheaf{I} \To \iota_\circ \iota^\circ \sheaf{I} = \bigoplus_{x \in X^0}  \constsheaf{\sheaf{I}}_x \oplus \bigoplus_{e \in X^1} \constsheaf{\sheaf{I}}_{\bar e}.$$
Here we write $\constsheaf{\sheaf{I}}_x$ for the skyscraper sheaf at $x$ with stalk $\sheaf{I}_x$.  And we write $\constsheaf{\sheaf{I}}_{\bar e}$ for the constant sheaf on $\bar e = \{ e, x,y \}$ with common stalk $\sheaf{I}_e$.  

Injectivity of $\sheaf{I}$ implies that $\sheaf{I}$ is a summand of the large direct sum; indecomposability of $\sheaf{I}$ implies that $\sheaf{I}$ is isomorphic to one of the summands.  This may be a bit subtle, since we make no finiteness assumptions on $\sheaf{I}$; but \cite[Theorem 8]{LeeLee} provides the necessary result, and we have
$$\sheaf{I} \isom \constsheaf{\sheaf{I}}_x \text{ or } \sheaf{I} \isom \constsheaf{\sheaf{I}}_{\bar e} \text{ for some vertex } x \text{ or edge } e.$$
Indecomposability now implies that the stalk $\sheaf{I}_x$ or $\sheaf{I}_e$ is one-dimensional, completing the proof.
\end{proof}

\subsection{Induction of sheaves}

To understand and construct equivariant sheaves, it is helpful to understand functors of restriction and induction.  Suppose that $H$ is a subgroup of $G$; then there is a forgetful functor $\res_H^G \From \Cat{Sh}_G(X) \To \Cat{Sh}_H(X)$.  In terms of presheaves, we have an inclusion of categories $\Cat{Q}_H \Into \Cat{Q}_G$, and thus a restriction of functors,
$$\res_H^G \From [\Cat{Q}_G, \Cat{Vec}] \To [\Cat{Q}_H, \Cat{Vec}].$$
This functor $\res_H^G$ has a left adjoint, which we call $\ind_H^G$, that is described in great generality at \cite[\href{https://stacks.math.columbia.edu/tag/00VC}{Tag 00VC}]{stacks-project}.  In our setting, we describe this {\em induction of equivariant sheaves}, beginning with an $H$-equivariant sheaf $\sheaf{T}$ on $X$.  For such a sheaf, the equivariant structure provides isomorphisms,
$$\rho(h,x) \From \sheaf{T}_x \To \sheaf{T}_{hx}$$
for all $h \in H$.  In this way, we may form the colimit
$$\sheaf{T}_{Hx} := \limdir_{h \in H} \sheaf{T}_{hx} = \bigoplus_{h \in H} T_{hx} / \langle t \sim \rho(h',hx) t \text{ for all } x \in X, h,h' \in H, t \in T_{hx} \rangle.$$
We say that $\sheaf{T}_{Hx}$ is the stalk of $\sheaf{T}$ on the $H$-orbit $Hx$.  By construction, $\sheaf{T}_{Hx}$ is canonically isomorphic to $\sheaf{T}_y$ for every $y \in Hx$.  The same procedure applies to construct $\sheaf{T}_{He}$ when $e$ is an edge.  

The sheaf $\sheaf{S} = \ind_H^G \sheaf{T}$ is now defined as follows.
\begin{enumerate}
\item
If $x \in X^0$ and $e \in X^1$, then
$$\sheaf{S}_x = \bigoplus_{\bar g \in G/H}  \sheaf{T}_{\bar g^{-1} x} \text{ and } \sheaf{S}_e = \bigoplus_{\bar g \in G/H}  \sheaf{T}_{\bar g^{-1} e}.$$
Here we note that if $\bar g = g H \in G/H$, then $\bar g^{-1} = H g^{-1} \in H \backslash G$, and $\bar g^{-1} x = H g^{-1} x$ is an $H$-orbit in $X$.
\item
If $x < e$, then the restriction maps $\sheaf{T}_{gx} \To \sheaf{T}_{ge}$ assemble to restriction maps $\sheaf{S}_x \To \sheaf{S}_e$.
\end{enumerate}

Note that the restriction maps operate separately on each summand in the definitions of $\sheaf{S}_x$ and $\sheaf{S}_e$; thus, in the non-equivariant category $\Cat{Sh}(X)$, we have an isomorphism of sheaves,
\begin{equation}
\label{ind-sheaf}
\sheaf{S} = \ind_H^G \sheaf{T}  \isom \bigoplus_{\bar g \in G/H} g^\ast \sheaf{T}.
\end{equation}
Here $g^\ast \sheaf{T}$ denotes the pullback with respect to the $G$-action on $X$, and each $g$ is a coset representative of $\bar g \in G/H$.

\begin{proposition}
The restriction functor $\res_H^G$ sends injective objects of $\Cat{Sh}_G(X)$ to injective objects of $\Cat{Sh}_H(X)$.  
\end{proposition}
\begin{proof}
Restriction $\res_H^G$ is right-adjoint to induction $\ind_H^G$.  Thus it suffices to prove that induction is exact.  But this is clear from the description of induction above in \eqref{ind-sheaf}; exactness is preserved by the direct sum and pullback shown there.  
\end{proof}

\begin{corollary}
Every injective $G$-equivariant sheaf on $X$ (injective object of $\Cat{Sh}_G(X)$) is a $G$-equivariant injective sheaf on $X$ (an injective object of $\Cat{Sh}(X)$ endowed with $G$-equivariant structure).
\end{corollary}
\begin{proof}
This follows from the above proposition, letting $H = \{ 1 \}$.
\end{proof}

\begin{proposition}
\label{inj-summand}
Let $\sheaf{S}$ be a $G$-equivariant sheaf on $X$.  Suppose that $\sheaf{S}$ has an injective summand, as a sheaf on $X$.  Then $\sheaf{S}$ has a $G$-equivariant injective summand.  In other words, we can write $\sheaf{S} = \sheaf{E} \oplus \sheaf{T}$ in the category $\Cat{Sh}(X)$, in such a way that $\sheaf{E}$ is a $G$-equivariant subsheaf of $\sheaf{S}$, and $\sheaf{E}$ is an injective object of $\Cat{Sh}(X)$.
\end{proposition}
\begin{proof}
Suppose there is an embedding $\eta \From \constsheaf{k}_x \Into \sheaf{S}$.  Let $E_x$ be the subspace of $\sheaf{S}_x$ spanned by $G_x \cdot \Im(\eta)$.  This defines a $G_x$-equivariant embedding of sheaves,
$$\eta \From \constsheaf{E}_x \Into \sheaf{S}.$$
Let $\sheaf{E} = \ind_{G_x}^G \constsheaf{E}_x$.  By adjointness, the embedding $\eta$ defines a morphism of $G$-equivariant sheaves,
$$\phi \From \sheaf{E} \Into \sheaf{S}.$$
Concretely, this $\phi$ transports the ($G_x$-equivariant) injective homomorphism $E \Into \sheaf{S}_x$ to an injective homomorphism $E \Into \sheaf{S}_{gx}$ for every $g \in G$, using the $G$-equivariant structure $\rho_{g,x}$.  The sheaf $\sheaf{E}$ is injective by construction, and hence is a summand of $\sheaf{S}$.  Hence we find that $\sheaf{S}$ has a $G$-equivariant injective summand.

Now suppose there is no such embedding of $\constsheaf{k}_x \Into \sheaf{S}$.  Thus there are no subsheaves of $\sheaf{S}$ supported on a single vertex.  Then, if $\sheaf{S}$ has an injective summand, it must admit an embedding $\eta \From \constsheaf{k}_{\bar e} \Into \sheaf{S}$.  Most of the construction is the same; let $E_e$ be the subspace of $\sheaf{S}_e$ spanned by $G_e \cdot \Im(\eta_e)$.  This defines a $G_e$-equivariant embedding of sheaves,
$$\eta \From \constsheaf{E}_{\bar e} \Into \sheaf{S}.$$

Now we let $\sheaf{E} = \ind_{G_e}^G \constsheaf{E}_{\bar e}$.  Adjointness defines the desired morphism of $G$-equivariant sheaves, $\phi \From \sheaf{E} \Into \sheaf{S}$.  This is obtained by summing the translates of $\eta$ via all elements of $G / G_e$.  In other words,
$$\phi = \bigoplus_{\bar g \in G / G_e} \phi_g \From \bigoplus_{\bar g \in G / G_e} g^\ast \constsheaf{E}_{\bar e} \To \sheaf{S},$$
where $\phi_g$ is the composition of $g^\ast \eta$ with the equivariance isomorphism $g^\ast \sheaf{S} \To \sheaf{S}$.  

We see that $\phi$ is injective because there can be no cancellations; there cannot be a nonzero $k_e \in \Ker(\phi)_e$ for any edge $e$, since the summands $g^\ast \constsheaf{E}_{\bar e}$ have no edges in their common support.  Thus $\Ker(\phi)$ is supported on vertices; it follows that $\Ker(\phi)$ is an injective subsheaf (hence summand) of the injective sheaf $\sheaf{E}$.  But no vertex-supported $\constsheaf{k}_x$ is a subsheaf of an edge-supported injective $\constsheaf{k}_{\bar e}$.  Thus $\Ker(\phi) = 0$.  Hence the $G$-equivariant injective sheaf $\sheaf{E}$ is a subsheaf, hence a summand of $\sheaf{S}$.
\end{proof}

\section{Collapse}

\subsection{Elementary collapse}

The idea of an elementary collapse has old roots, in the work of Whitehead from 1938 \cite{Whitehead}.  There, the idea is that one looks for simplices $\sigma, \tau$ in a simplicial complex, for which $\sigma$ is a codimension-one face of a maximal simplex $\tau$, and no other maximal simplex contains $\sigma$.  In this case, one can ``collapse'' the complex by removing the pair $\{ \sigma, \tau \}$.  The figure below gives a sequence of elementary collapses.
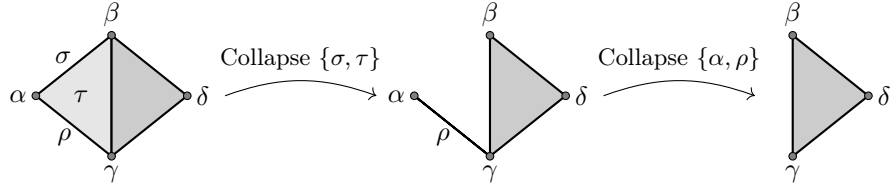
\begin{center}
\begin{tikzpicture}
\begin{scope}[xshift=-4.5cm]
\coordinate (A) at (-1,0); 
\coordinate (B) at (0,0.8); 
\coordinate (C) at (0,-0.8); 
\coordinate (D) at (1,0); 
\filldraw[draw=black, thick, fill=black!10] (A) to node[above left=-0.1cm] {$\sigma$} (B) -- (C) to node[below left=-0.1cm] {$\rho$} (A);
\filldraw[draw=black, thick, fill=black!20] (D) -- (B) -- (C) -- cycle;
\filldraw[draw=black, fill=gray] (A) circle (0.05) node[left] {$\alpha$};
\filldraw[draw=black, fill=gray] (B) circle (0.05)node[above]  {$\beta$};
\filldraw[draw=black, fill=gray] (C) circle (0.05) node[below]  {$\gamma$};
\filldraw[draw=black, fill=gray] (D) circle (0.05) node[right]  {$\delta$};
\node (T) at (-0.4,0) {$\tau$};
\end{scope}
\draw[->] (-3,0) to [bend left=20] node[above] {\small Collapse $\{\sigma, \tau \}$} (-1,0);
\begin{scope}[xshift=0.5cm]
\coordinate (A) at (-1,0); 
\coordinate (B) at (0,0.8); 
\coordinate (C) at (0,-0.8);
\coordinate (D) at (1,0); 
\draw[draw=black, thick] (A) -- (C) to node[below left=-0.1cm] {$\rho$} (A);
\filldraw[draw=black, thick, fill=black!20] (D) -- (B) -- (C) -- cycle;
\filldraw[draw=black, fill=gray] (A) circle (0.05) node[left] {$\alpha$};
\filldraw[draw=black, fill=gray] (B) circle (0.05)node[above]  {$\beta$};
\filldraw[draw=black, fill=gray] (C) circle (0.05) node[below]  {$\gamma$};
\filldraw[draw=black, fill=gray] (D) circle (0.05) node[right]  {$\delta$};
\end{scope}
\draw[->] (2,0) to [bend left=20] node[above] {\small Collapse $\{\alpha, \rho \}$} (4,0);
\begin{scope}[xshift=4.5cm]
\coordinate (B) at (0,0.8); 
\coordinate (C) at (0,-0.8); 
\coordinate (D) at (1,0); 
\filldraw[draw=black, thick, fill=black!20] (D) -- (B) -- (C) -- cycle;
\filldraw[draw=black, fill=gray] (B) circle (0.05)node[above]  {$\beta$};
\filldraw[draw=black, fill=gray] (C) circle (0.05) node[below]  {$\gamma$};
\filldraw[draw=black, fill=gray] (D) circle (0.05) node[right]  {$\delta$};
\end{scope}
\end{tikzpicture}
\end{center}

We introduce an analogue here for (equivariant) sheaves on the tree $X$.
\begin{definition}
A {\em basic elementary sheaf} on $X$ is a sheaf of the form $\constsheaf{k}_{x,e}$ for some $x < e$; in other words $\constsheaf{k}_{x,e}$ is the sheaf whose stalks vanish outside of one vertex $x$ and one edge $e > x$, and for which
$$(\constsheaf{k}_{x,e})_x = (\constsheaf{k}_{x,e})_e = k \text{ and } \Res_{x,e} = \Id.$$
More generally, an {\em elementary sheaf} on $X$ is a sheaf which is isomorphic to a direct sum of basic elementary sheaves.
\end{definition}

When $\sheaf{J}$ is an elementary sheaf, we may decompose:  $\sheaf{J} = \bigoplus_{x < e} \sheaf{J}^{x,e}$.  Here $\sheaf{J}^{x,e}$ is a direct sum of elementary subsheaves supported on the pair $\{ x,e \}$, thus a constant sheaf on $\{ x,e \}$ (extended by zero elsewhere).  More canonically, we can identify
$$\sheaf{J}^{x,e} \isom \Hom(\constsheaf{k}_{x,e}, \sheaf{J}) \otimes \constsheaf{k}_{x,e}.$$

If $\sheaf{S}$ is a sheaf on $X$, an elementary subsheaf of $\sheaf{S}$ will mean an embedding $\sheaf{J} \Into \sheaf{S}$ of an elementary sheaf.  In this case, we say that the morphism from $\sheaf{S}$ onto its quotient $\sheaf{S} / \sheaf{J}$ is an {\em elementary collapse}.  More generally, a {\em collapse} of $\sheaf{S}$ is a morphism of sheaves $\sheaf{S} \To \sheaf{T}$ which is composed from a finite sequence of elementary collapses.

In the traditional setting of simplicial complexes, an elementary collapse is a particular type of homotopy equivalence.  We do not explore the homotopy theory of sheaves on $X$, but we do note that our collapses induce isomorphisms in cohomology.
\begin{proposition}
Let $\sheaf{S} \Onto \sheaf{S} / \sheaf{J}$ be an elementary collapse of $\sheaf{S}$.  Then the natural maps $H_c^p(X, \sheaf{S}) \To H_c^p(X, \sheaf{S} / \sheaf{J})$ are isomorphisms for all $p \geq 0$.
\end{proposition}
\begin{proof}
Using the long exact sequence in cohomology, it suffices to see that $H_c^p(X, \sheaf{J}) = 0$ for all $p \geq 0$.  But this is clear, as the cohomology of a basic elementary sheaf is computed from the complex $\cdots \To 0 \To k \xrightarrow{\Id} k \To 0 \To \cdots$.  
\end{proof}

\subsection{Equivariant collapse}

In order to collapse equivariant sheaves, we need an equivariant notion of an elementary sheaf.  We take the following definition.
\begin{definition}
A $G$-equivariant elementary  sheaf on $X$ is an equivariant sheaf $\sheaf{J}$, which as a sheaf (forgetting the equivariant structure) is an elementary sheaf.
\end{definition}

If $\sheaf{S}$ is a $G$-equivariant sheaf on $X$, and $\sheaf{J}$ is an equivariant elementary subsheaf of $\sheaf{S}$, we say that $\sheaf{S}$ {\em equivariantly collapses} onto its quotient $\sheaf{S} / \sheaf{J}$.  As in the non-equivariant setting, the natural map provides an isomorphism in cohomology,
$$H_c^p(X, \sheaf{S}) \xrightarrow{\sim} H_c^p(X, \sheaf{S} / \sheaf{J}).$$
These are now isomorphisms in the category $\Cat{Rep}(G)$.

\subsection{The maximal elementary subsheaf}

Let $\sheaf{S}$ be a sheaf on $X$.  Recall from Corollary \ref{indec-inj} that $\sheaf{S}$ has an injective subsheaf (hence summand) iff there exists a vertex $x$ such that $k_{x} \Into \sheaf{S}$ or there exists an edge $e$ such that $k_{\bar e} \Into \sheaf{S}$.  In our previous paper \cite{MWInd}, we introduced functorially-defined subsheaves $\sheaf{S}^{\el} \subset \sheaf{S}$ and  $\sheaf{S}^{\uni} \subset \sheaf{S}$, whose definitions we recall here.

\begin{definition}
Given a vertex $x \in X^0$, and element $f \in \sheaf{S}_x$, we say that $f$ is {\em elliptic} if $\Res_{x,e} f = 0$ for all edges $e > x$.  We write $\sheaf{S}_x^{\el}$ for the subspace of elliptic vectors in $\sheaf{S}_x$.  This defines a subsheaf $\sheaf{S}^{\el} \subset \sheaf{S}$ whose support lies within the set $X^0$ of vertices of $\sheaf{S}$.
\end{definition}
Note that the sheaf $\sheaf{S}^{\el}$ is injective, $\sheaf{S}^{\el} = \bigoplus_x \constsheaf{\sheaf{S}}_x^{\el}$, and thus a summand of $\sheaf{S}$.

\begin{definition}
Given a vertex $x \in X^0$, and element $f \in \sheaf{S}_x$, we say that $f$ is {\em unifacial} if there exists a unique edge $e > x$ for which $\Res_{x,e} f \neq 0$.  We write $\sheaf{S}_x^{\uni}$ be the subspace of $\sheaf{S}_x$ generated by unifacial vectors.  Finally, when $e = \edge{xy}$, we let 
$$\sheaf{S}_e^{\uni} = \Res_{x,e} \sheaf{S}_x^{\uni} + \Res_{y,e} \sheaf{S}_y^{\uni}.$$
This defines a subsheaf $\sheaf{S}^{\uni} \subset \sheaf{S}$.
\end{definition}

A priori, $\sheaf{S}_x^{\uni}$ might contain a nonzero elliptic vector; one might have $f_1 \neq f_2 \in \sheaf{S}_x^{\uni}$ while $\Res_{x,e} f_1 = \Res_{x,e} f_2$, leading to a cancellation $f_1 - f_2 \in \sheaf{S}_x^{\el}$.  But when $\sheaf{S}$ has no injective summands, such cancellation cannot occur.  The following theorem generalizes this idea. 

\begin{thm}
\label{uni-elementary}
Suppose that $\sheaf{S}$ is a sheaf on $X$ with no injective summands.  Then $\sheaf{S}^{\uni}$ is the unique maximal elementary subsheaf of $\sheaf{S}$.  
\end{thm}
\begin{proof}
First, suppose that $\sheaf{J} = \bigoplus_{x < e} \sheaf{J}^{x,e}$ is an elementary subsheaf of $\sheaf{S}$.  Then every nonzero germ $f \in \sheaf{J}_x^{x,e}$ is unifacial; in this way, we see that $\sheaf{J}^{x,e} \subset \sheaf{S}^{\uni}$ for all $x < e$.  It follows that $\sheaf{J} \subset \sheaf{S}^{\uni}$.  

Thus it remains to show that $\sheaf{S}^{\uni}$ itself is an elementary subsheaf of $\sheaf{S}$.  For this, consider the following subsheaf $\sheaf{S}^{x,e} \subset \sheaf{S}^{\uni}$ supported on $\{ x,e \}$, for all $x < e$:
$$\sheaf{S}_x^{x,e} = \Span_k \{ f \in \sheaf{S}_x^{\uni} : \Res_{x,e} f \neq 0 \} \text{ and } \sheaf{S}_e^{x,e} = \Res_{x,e} \sheaf{S}_x^{x,e}.$$
First we claim that $\sheaf{S}_x^{x,e}$ is elementary, i.e., restriction $\sheaf{S}_x^{x,e} \To \sheaf{S}_e^{x,e}$ is an isomorphism. This map is surjective by the construction above.  If it were not injective, then there would exist $f \in \sheaf{S}_x^{x,e}$ satisfying $\Res_{x,e} f = 0$.  But such an $f$ would define an embedding $\constsheaf{k}_x \Into \sheaf{S}$ of an injective subsheaf, contradicting our no-injective-summand hypothesis.

Now that we have shown each $\sheaf{S}^{x,e}$ is an elementary subsheaf of $\sheaf{S}$, we claim that the summation map,
$$\Sigma \From \bigoplus_{x < e} \sheaf{S}^{x,e} \To \sheaf{S}$$
is injective.  (Note that local finiteness is required here to make sense of the sum!)  If not, then there would exist a vertex or edge at which the summation map is not injective.  If at an edge $e = \edge{xy}$, then the summation map
$$\Sigma \From \sheaf{S}_e^{x,e} \oplus \sheaf{S}_e^{y,e} \To \sheaf{S}_e$$
would not be injective.  This would give germs $f_x \in \sheaf{S}_x^{x,e}$ and $f_y \in \sheaf{S}_y^{y,e}$ with $\Res_{x,e} f_x = \Res_{y,e} f_y$.  And this in turn would define an embedding $\constsheaf{k}_{\bar e} \Into \sheaf{S}$ of an injective subsheaf, contradicting our no-injective-summand hypothesis.

If at a vertex $x$, then the summation map,
$$\Sigma \From \bigoplus_{e > x} \sheaf{S}_x^{x,e} \To \sheaf{S}_x$$
would not be injective.  But germs that propagate in distinct directions cannot cancel.  

Thus we find that summation gives an injective map $\bigoplus_{x < e} \sheaf{S}^{x,e} \To \sheaf{S}$.  By construction, its image is precisely $\sheaf{S}^{\uni}$.  Since each $\sheaf{S}^{x,e}$ is an elementary subsheaf, the direct sum is elementary, completing the proof.
\end{proof}

When $\sheaf{S}$ is a sheaf on $X$ with no injective summands, it follows that the map
$$\coll \From \sheaf{S} \Onto \sheaf{S}' := \sheaf{S} / \sheaf{S}^{\uni}$$
is the largest elementary collapse of $\sheaf{S}$ (the one which has maximal kernel).  Note that if $\sheaf{S}$ is a $G$-equivariant sheaf on $X$, then $\sheaf{S}^{\uni}$ is  a $G$-equivariant subsheaf, and thus $\coll$ is a $G$-equivariant collapse.  This induces an isomorphism (of $G$-representations in the equivariant case),
$$\coll \From H_c^p(X, \sheaf{S}) \xrightarrow{\sim} H_c^p(X, \sheaf{S}') \text{ for } p = 0, 1.$$

The previous theorem does not exclude a degenerate case -- when $\sheaf{S}^{\uni}$ is zero, and thus the collapse $\coll$ is the identity.  But the following guarantees collapse is nontrivial when our sheaf has a compactly-supported section.
\begin{proposition}
\label{Suni-nonzero}
Suppose that $\sheaf{S}$ is a sheaf on $X$ with no injective summands and $H_c^0(X, \sheaf{S}) \neq 0$.  Then $\sheaf{S}^{\uni} \neq 0$.
\end{proposition}
\begin{proof}
This is essentially proven in \cite{MWInd}, but we quickly review the argument.  Begin with a nonzero $f \in H_c^0(X, \sheaf{S})$.  Then $\Supp(f)$ is a nonempty finite closed subset of $X$, thus a finite union of finite trees.  Since $\sheaf{S}$ has no injective summands, $\Supp(f)$ has no isolated points.  Thus $\Supp(f)$ contains a leaf $x$, with a unique edge $e > x$ satisfying $e \in \Supp(f)$.  Hence $0 \neq f_x \in \sheaf{S}_x^{x,e} \subset \sheaf{S}_x^{\uni}$.  
\end{proof}

\subsection{Support radius and collapse}
Suppose that $\sheaf{S}$ is a sheaf on $X$.  We will be interested in sections of $\sheaf{S}$ supported on balls in $X$.  Here, if $x \in X^0$ is a vertex, we write $B(x; r)$ for the closed ball of radius $r$ centered at $x$.  Thus $B(x; 0) = \{ x \}$, and $B(x; 1)$ includes $x$, all edges $e > x$, and all endpoints of those edges.  We also allow balls centered at edges:  if $e = \edge{xy}$, then $B(e; r)$ is simply $B(x;r) \cup B(y;r) \cup \{ e \}$.  Thus $B(e;0) = \bar e = \{ x,y,e \}$.   

We introduce a metric for sections of $\sheaf{S}$ here.
\begin{definition}
If $f \in H_c^0(X, \sheaf{S})$, the {\em support radius} $\rho(f)$ of $f$ is the smallest integer $r$ such that $\Supp(f)$ is contained in a ball of radius $r$.  The {\em support radius} of the sheaf $\sheaf{S}$ is defined by
$$\rho(\sheaf{S}) = \min \{ \rho(f) : 0 \neq f \in H_c^0(X, \sheaf{S}) \}.$$
If $H_c^0(X, \sheaf{S}) = 0$, we write $\rho(\sheaf{S}) = \infty$.
\end{definition}

\begin{lemma}
\label{rho-inj}
The sheaf $\sheaf{S}$ satisfies $\rho(\sheaf{S}) = 0$ if and only if $\sheaf{S}$ has an injective summand.
\end{lemma}
\begin{proof}
Note that if $f \in H_c^0(X, \sheaf{S})$, and $\rho(f) = 0$, then $\Supp(f) = \{ x \}$ or $\Supp(f) = \bar e$, for some vertex $x$ or some edge $e$.  Hence $f$ defines an injective subsheaf (hence summand) of $\sheaf{S}$.  Conversely, an indecomposable injective subsheaf of $\sheaf{S}$ defines a section $f$ with $\rho(f) = 0$.
\end{proof}

If $B$ is any compact subset of $X$, let $H_B^0(X, \sheaf{S})$ be the space of sections whose support lies within $B$.  If $G_B$ is the stabilizer of $B$ (set-wise, not point-wise), and $\sheaf{S}$ is a $G$-equivariant sheaf, then $H_B^0(X, \sheaf{S})$ is naturally a representation of $G_B$.  In particular, if $B = B(x;r)$ or $B = B(e;r)$, then $H_B^0(X, \sheaf{S})$ is a representation of $G_x$ or $G_e$, respectively.  The following theorem allows us to track the support of global sections, as we collapse a sheaf.
\begin{thm}
\label{track-support}
Suppose that $\sheaf{S}$ is a sheaf with $\rho(\sheaf{S})$ finite and nonzero.  Let $\sheaf{S}' = \sheaf{S} / \sheaf{S}^{\uni}$.  Let $B = B(x;r)$ or $B = B(e;r)$ be a ball of positive radius $r$ in $X$.  Let $B' = B(x;r-1)$ or $B' = B(e;r-1)$ be the ball of radius $r-1$ with the same central vertex or edge as $B$.  Then the collapse isomorphism $\coll \From H_c^0(X, \sheaf{S}) \To H_c^0(X, \sheaf{S}')$ restricts to an isomorphism,
$$\coll_B \From H_B^0(X, \sheaf{S}) \To H_{B'}^0(X, \sheaf{S}').$$
Moreover, for every $f \in H_c^0(X, \sheaf{S})$ with $f' = \coll(f)$, we have
$$\rho(f') = \rho(f) - 1.$$
\end{thm}
\begin{proof}
Let $r$ be a positive integer, and consider a closed ball $B = B(x;r)$ or $B = B(e;r)$ in $X$.  Begin with an element $f \in H_B^0(X, \sheaf{S})$.  Then for every leaf $x$ of $B$, we find that $f_x \in \sheaf{S}_x^{\uni}$.  Indeed, since $\Supp(f) \subset B$, the germ $f_x$ can only satisfy $\Res_{x,e} f_x \neq 0$ when $e$ is an edge in $B$.  Therefore, the image $f' \in H_c^0(X, \sheaf{S}')$ vanishes at the leaf $x$.  Hence $\Supp(f') \subset B(x;r-1)$ or $\Supp(f') \subset B(e;r-1)$.  This proves that the collapse isomorphism $\coll$ restricts to a map,
$$\coll_B \From H_B^0(X, \sheaf{S}) \To H_{B'}^0(X, \sheaf{S}').$$
This map $\coll_B$ is injective (as the restriction of an isomorphism).  Moreover, if $B$ is a minimal ball containing the support of $f$, so $r = \rho(f)$, then $\Supp(f') \subset B'$ and 
$$\rho(f') \leq \rho(f) - 1.$$

It remains to show that $\coll_B$ is surjective, and to prove the reverse inequality.  For this, let us begin with a section $h' \in H_{B'}^0(X, \sheaf{S}')$.  Then there exists a unique section $h \in H_c^0(X, \sheaf{S})$ with $\coll(h) = h'$.  Suppose that $z \in \Supp(h)$.  If $z \not \in B'$, then $h_z \neq 0$ and $h_z' = 0$, so it follows that $h_z \in \sheaf{S}_z^{\uni}$.  Since $\sheaf{S}$ has no injective summands, we may decompose $h_z$ uniquely as
$$h_z = \sum_{e > z} h_{z,e} \text{ with } h_{z,e} \in \sheaf{S}_z^{z,e}.$$
If $e = \edge{zy}$, then we have $\Res_{z,e} h_{z,e} = \Res_{y,e} h_y$.  If $h_{z,e} \neq 0$, then $h_y \neq 0$.

If $y \not \in B'$, then $h_y' = 0$, so $h_y \in \sheaf{S}_y^{\uni}$.  It follows that the pair $h_{z,e}, h_y$ defines an embedding $k_{\bar e} \Into \sheaf{S}$, contradicting the fact that $\sheaf{S}$ has no injective summands.  Therefore, we find that $h_{z,e}$ is nonzero only if $y \in B'$.  This is only possible if $z \in B$; geometrically, we are using the fact that the vertices of $B$ are the vertices of $B'$ together with vertices that are a connected via a one-edge path to vertices of $B'$.  

From this it follows that $\Supp(h) \subset B$.  This completes the proof that $\coll_B$ is surjective, thus an isomorphism.  We find that $\rho(h) \leq \rho(h') + 1$ by this argument.  Hence (taking $h' = f'$ from earlier), we find that $\rho(f) \leq \rho(f') + 1$, completing the proof.
\end{proof}

\section{Collapse and geometrically minimal K-types}

In this section, we fix a group $G$ acting on a tree $X$ as before, and a plumbing system $\NN$.  We take $(\pi, V)$ to be an irreducible, $\NN$-compact, $\NN$-cogenerated, and $\NN$-characteristic representation of $G$.  

For example, $G$ may be $p$-adic group of relative rank one, or a ``covering'' of such a group, or a disconnected group, as discussed in the earlier Examples \ref{Ex1}, \ref{Ex2}, \ref{Ex3}.  And in these cases, by Theorem \ref{cusp-cpt}, we may take any irreducible supercuspidal representation, with coefficients in a field $k$ with $p \nmid \Char(k)$.  The depth of the representation should be bounded by $\NN$, in the sense that $V^{N_x} \neq 0$ for some vertex $x \in X^0$.  Finally, we let $\sheaf{S}$ be the localization, $\sheaf{S} = \loc_\NN(\pi, V)$.  Thus $\sheaf{S}$ is a $G$-equivariant sheaf on $X$, and localization gives an identification,
$$\loc \From V \xrightarrow{\sim} H_c^0(X, \sheaf{S}).$$

In this section, we demonstrate that $\sheaf{S}$ can be equivariantly collapsed onto an injective sheaf.  Tracking supports through the collapse, we find a ``geometrically minimal $K$-type'' $(\rho, W)$ within $V$.  This $K$-type occurs with multiplicity one, and $(\pi, V)$ is naturally identified with $\cInd_K^G W$ .  

\subsection{The collapse sequence}

If $\sheaf{S}$ has no injective summands, then $\sheaf{S}^{\uni}$ is the maximal elementary subsheaf of $\sheaf{S}$, and we define $\sheaf{S}' = \sheaf{S} / \sheaf{S}^{\uni}$.  Continuing, if $\sheaf{S}'$ has no injective summands, we write $\sheaf{S}'' = \sheaf{S}' / (\sheaf{S}')^{\uni}$, etc., producing a sequence of $G$-equivariant sheaves and collapses,
$$\sheaf{S} \Onto \sheaf{S}' \Onto \sheaf{S}'' \Onto \cdots \Onto \sheaf{S}^{(r)}.$$
Here we continue until we reach a sheaf $\sheaf{S}^{(r)}$ which has an injective summand.  This will be called the collapse sequence of $\sheaf{S}$; a priori, it may not terminate.  But since $V = H_c^0(X, \sheaf{S})$ is an irreducible representation of $G$, we have the following.
\begin{thm}
Let $r = \rho(\sheaf{S})$ denotes the support radius of $\sheaf{S}$.  Then the collapse sequence of $\sheaf{S}$ terminates at $\sheaf{S}^{(r)}$.  Moreover, this final term $\sheaf{S}^{(r)}$ is an indecomposable $G$-equivariant injective sheaf.  To clarify, $\sheaf{S}^{(r)}$ is an injective object of $\Cat{Sh}(X)$, and $\sheaf{S}^{(r)}$ is indecomposable as an object of $\Cat{Sh}_G(X)$.
\end{thm}
\begin{proof}
Let $B = B(x;r)$ or $B = B(e;r)$ be a minimal ball which supports a section of $\sheaf{S}$.  If $r > 0$, then $\sheaf{S}$ has no injective summands, and collapse provides an isomorphism (by Theorem \ref{track-support})
$$\coll_B \From H_B(X, \sheaf{S}) \To H_{B'}(X, \sheaf{S}').$$
In particular, we find that $\rho(\sheaf{S}') \leq r - 1$.  Conversely, any section $h' \in H_c^0(X, \sheaf{S}')$ supported on a ball of radius $s$ arises from collapsing a section $h \in H_c^0(X, \sheaf{S})$ supported on a ball of radius $s+1$.  This demonstrates that $\rho(\sheaf{S}') = \rho(\sheaf{S}) - 1$.  

Continuing in this way, we find that the support radius decreases by one at each step in the collapse sequence
$$\sheaf{S} \Onto \sheaf{S}' \Onto \sheaf{S}'' \Onto \cdots \Onto \sheaf{S}^{(r)},$$
and the first occurrence of an injective summand occurs after $r = \rho(\sheaf{S})$ steps.  

We are left to consider the final sheaf $\sheaf{S}^{(r)}$ in the chain.  As $\rho(\sheaf{S}^{(r)}) = 0$, the sheaf $\sheaf{S}^{(r)}$ has an injective summand.  Moreover, as collapse does not change cohomology, we still find an isomorphism of $G$-representations,
$$V \xrightarrow{\sim} H_c^0(X, \sheaf{S}^{(r)}).$$
It remains to prove that $\sheaf{S}^{(r)}$ is an indecomposable $G$-equivariant injective sheaf.  

By Proposition \ref{inj-summand}, since $\sheaf{S}^{(r)}$ has an injective summand as a sheaf, it has a $G$-equivariant injective summand.  Thus we can find a decomposition in $\Cat{Sh}(X)$,
$$\sheaf{S}^{(r)} = \sheaf{E} \oplus \sheaf{T},$$
where $\sheaf{E}$ is an injective sheaf, and a $G$-equivariant subsheaf of $\sheaf{S}^{(r)}$.  Taking global sections, we find
$$V \isom H_c^0(X, \sheaf{S}^{(r)}) \isom H_c^0(X, \sheaf{E}) \oplus H_c^0(X, \sheaf{T}),$$
and $H_c^0(X, \sheaf{E})$ is a $G$-subrepresentation of $V$.  Since $\sheaf{S}$ is generated by compactly-supported global sections, the same is true for its quotient $\sheaf{S}^{(r)}$, and thus for its summands $\sheaf{E}$ and $\sheaf{T}$.  Since $(\pi, V)$ is an irreducible representation of $G$, and $\sheaf{E}$ is nonzero, it follows that $\sheaf{T} = 0$.

Finally, we claim that $\sheaf{E}$ is indecomposable, as an object of $\Cat{Sh}_G(X)$.  For if it decomposed as a direct sum of nonzero $G$-equivariant sheaves, $\sheaf{E} = \sheaf{E}_1 \oplus \sheaf{E}_2$, these summands would be generated by their compactly-supported global sections.  This would yield a nontrivial decomposition of $V = H_c^0(X, \sheaf{S}^{(r)})$, a contradiction as before.  Therefore $\sheaf{S}^{(r)} = \sheaf{E}$ is an indecomposable object of $\Cat{Sh}_G(X)$ and an injective object of $\Cat{Sh}(X)$. 
\end{proof}

\subsection{$K$-types and induction}

We keep the notation from before, and now track sections of the sheaf $\sheaf{S}$ as they pass through the collapse sequence.  Let $r = \rho(\sheaf{S})$ as before.  A {\em minimal support-ball} for $\sheaf{S}$ will be a minimal ball $B$ which supports a section.  Thus $B = B(x;r)$ or $B = B(e;r)$ for some vertex $x$ or edge $e$.  Note that if $x < e$, then $B(x;r) \subset B(e;r)$.  Thus if $B(x;r)$ is a minimal support-ball for $\sheaf{S}$, our definition states that $B(e;r)$ is {\em not} a minimal support-ball, even though their radii are the same.  

\begin{definition}
Let $B$ be a minimal support-ball for the sheaf $\sheaf{S}$.  Let $K = G_B$ be its stabilizer, and $H_B^0(X, \sheaf{S})$ the resulting nonzero representation of $K$.  The pair $(K, H_B^0(X, \sheaf{S}))$ is called a {\em geometrically minimal $K$-type} for $(\pi, V)$.
\end{definition}

The following theorem demonstrates that geometrically minimal $K$-types provide inducing data.
\begin{thm}
Let $B$ be a minimal support-ball for the sheaf $\sheaf{S}$, centered at a vertex $x$ or edge $e$.  Let $K = G_B$ and let $W = H_B^0(X, \sheaf{S})$ be the geometrically minimal $K$-type.  Then $W$ is an irreducible representation of $K$, and the inclusion of $K$-representations $W \Into V$ extends to an isomorphism of $G$-representations:  $\cInd_K^G W \xrightarrow{\sim} V$.
\end{thm}
\begin{proof}
Let $r = \rho(\sheaf{S})$, so $B$ is a ball of radius $r$.  Consider the chain of balls with the same center as $B$ but decreasing radii,
$$B \supset B' \supset B'' \supset \cdots \supset B^{(r)}.$$
Thus $B^{(r)} = \{ x \}$ or $B^{(r)} = \bar e$ is the ball of radius zero.  By Theorem \ref{track-support}, collapse provides isomorphisms,
$$W = H_B^0(X, \sheaf{S}) \xrightarrow{\sim} H_{B'}^0(X, \sheaf{S}') \xrightarrow{\sim} \cdots \xrightarrow{\sim} H_{B^{(r)}}^0(X, \sheaf{S}^{(r)}).$$
At the end, $\sheaf{S}^{(r)}$ is an indecomposable $G$-equivariant injective sheaf on $X$.  The identification of $W$ with $H_{B^{(r)}}^0(X, \sheaf{S}^{(r)})$ gives an embedding of $K$-equivariant sheaves, $\constsheaf{W}_x \Into \sheaf{S}^{(r)}$ or $\constsheaf{W}_{\bar e} \Into \sheaf{S}^{(r)}$.  By adjointness (cf.~Proposition \ref{inj-summand}), this gives an embedding of $G$-equivariant sheaves,
$$\ind_K^G \constsheaf{W}_{x} \Into \sheaf{S}^{(r)} \text{ or } \ind_K^G \constsheaf{W}_{\bar e} \Into \sheaf{S}^{(r)}.$$
Taking compactly-supported global sections, we find an inclusion,
$$\cInd_K^G W \Into H_c^0(X, \sheaf{S}^{(r)}).$$
But collapse provides isomorphisms,
$$V = H_c^0(X, \sheaf{S}) \xrightarrow{\sim} H_c^0(X, \sheaf{S}') \xrightarrow{\sim} \cdots \xrightarrow{\sim} H_c^0(X, \sheaf{S}^{(r)}).$$
Since $V$ is an irreducible representation of $G$, we find that the inclusion above gives an isomorphism,
$$\cInd_K^G W \xrightarrow{\sim} V.$$
Since $(\pi, V)$ is an irreducible representation of $G$, it follows that $W$ is an irreducible representation of $K$ that occurs with multiplicity one in $V$.  
\end{proof}

\section*{Appendix}

The following prompt was input into Claude Fable 5, on June 9, 2026.  See \cite{Claude} for more details, including the full chat transcript.  For context, this chat was an attempt to assess the technical capabilities of Fable 5, by asking a variety of math and programming questions -- including the math question here which had been a minor thorn in our side.
\begin{quote}
Alright good!  Now let's test your math solving capability.  Let $k$ be a field, and let $G$ be a finite group.  Let $(\pi, V)$ be a finite-dimensional representation of $G$, so $V$ is a finite-dimensional $k$-vector space, and $\pi$ is a homomorphism from $G$ to $GL(V)$.  Let's say that $V$ is ``$G$-nice'' if the canonical homomorphism from $V^G$ ($G$-invariants) to $V_G$ ($G$-coinvariants) is an isomorphism.  So every representation will be $G$-nice when the characteristic of $k$ does not divide the cardinality of $G$.  Describe the $G$-nice representations.  As a category?  An easy-to-check criterion that goes beyond the whole ``avoid bad field characteristic'' criterion?  A link to something else well known to representation theorists?  Let's see what you come up with.
\end{quote}

The LLM response was essentially the content and full proof of our Proposition \ref{charchar}, with some illustrative examples.  For example, Claude Fable told us that the trivial representatiion lies buried in the middle of the projective cover of the sign representation of $S_3$, in characteristic $p=3$.  Thus this projective cover qualifies as $S_3$-characteristic, since it has no trivial representation in its head or socle.  While we are certain such examples are well-known to those who work in modular representation theory, that is not our field of expertise.  We hope the insights of Claude Fable 5 are helpful for the reader, even if they are a very minor part of this paper.

% -------------------------------------------0---------------------
\bibliographystyle{amsalpha}
\bibliography{SCS}
% ----------------------------------------------------------------

\end{document}